\documentclass[preprint,11pt,3p]{elsarticle}

\usepackage{times}
\usepackage{color}
\usepackage[colorlinks=true]{hyperref}
\usepackage{amsmath,amssymb,amsthm,amscd,mathrsfs}
\numberwithin{equation}{section}

\newtheorem{theorem}{Theorem}[section]

\newtheorem{proposition}{Proposition}[section]
\newtheorem{lemma}{Lemma}[section]
\newtheorem{definition}{Definition}[section]

\biboptions{sort&compress}

\journal{Nonlinearity}

\begin{document}

\begin{frontmatter}

\title{Multifractal analysis of the maximal product of consecutive partial quotients in continued fractions}
\author[label1]{Kunkun Song}
\ead{songkunkun@hunnu.edu.cn}
\author[label2]{Dingding Yu\corref{cor1}}
\ead{Yudding\_sgr@whu.edu.cn}
\author[label2]{Yueli Yu}
\ead{yuyueli@whu.edu.cn}

\address[label1]{Key Laboratory of Computing and Stochastic Mathematics (Ministry of Education),
School of Mathematics and Statistics,
Hunan Normal University, Changsha, 410081, China}
\address[label2]{School of Mathematics and Statistics, Wuhan University, Wuhan 430072, China}

\cortext[cor1]{corresponding author}

\begin{abstract}
Let $[a_1(x), a_2(x), \ldots, a_n(x), \ldots]$ be the continued fraction expansion of an irrational number $x\in (0,1)$. We study the growth rate of the maximal product of consecutive partial quotients among the first $n$ terms, defined by $L_n(x)=\max_{1\leq i\leq n}\{a_i(x)a_{i+1}(x)\}$, from the viewpoint of multifractal analysis. More precisely, we determine
the Hausdorff dimension of the level set 
\[L(\varphi):=\left\{x\in (0,1):\lim_{n\to \infty}\frac{L_n(x)}{\varphi(n)}=1\right\},\]
where $\varphi:\mathbb{R^+}\to\mathbb{R^+}$ is an increasing function such that $\log \varphi$ is a regularly increasing function with index $\rho$. We show that there exists a jump of the Hausdorff dimension of $L(\varphi)$ when $\rho=1/2$. We also construct uncountably many discontinuous functions $\psi$ that cause the Hausdorff dimension of $L(\psi)$  to transition continuously from 1 to 1/2, filling the gap when $\rho=1/2$.
\end{abstract}
\end{frontmatter}

\section{Introduction}
Each irrational number $x\in(0,1)$ admits a unique infinite continued fraction expansion given by
\begin{equation*}
  x=\cfrac{1}{a_1(x)+\cfrac{1}{a_2(x)+\ddots+\cfrac{1}{a_n(x)+\ddots}}} :=[a_1(x),a_2(x),\ldots,a_n(x),\ldots],
\end{equation*}
where $a_1(x), \ldots, a_n(x),\ldots$ are positive integers, called the partial quotients of $x$. For any $n\geq1$, let 
\begin{equation}\label{jd}
\frac{p_n(x)}{p_n(x)}:=\left[a_1(x),\ldots , a_n(x)\right]
\end{equation}
be the $n$-th convergent of $x$. For basic properties of continued fractions, we refer to \cite{iosifescu2002metrical, Khi64} and the references therein.

The theory of continued fractions is closely related to the theory of Diophantine approximation, which studies how well a real number can be approximated by rational numbers. The approximation rate of the sequence of convergents is described by
\[\frac{1}{3a_{n+1}(x)q_n^2(x)}
\leq \left|x-\frac{p_n(x)}{q_n(x)}\right|
\leq \frac{1}{a_{n+1}(x)q_n^2(x)}.\]
This indicates that
the asymptotic Diophantine properties of $x\in(0,1)$ are reflected in the growth rate of its partial quotients. Regarding the uniform Diophantine properties, the first result is Dirichlet's theorem.

\begin{theorem}(Dirichlet, 1842)
  For any $x\in(0,1)$ and $t>1$, there exists %non-trivial solutions 
  $(p,q)\in\mathbb{N}^2$ such that
  $$|qx-p|\le 1/t~~\text{and}~~1\le q<t.$$
\end{theorem}
 It follows that for any $x\in(0,1)$, there exist infinitely many solutions $(p,q)\in\mathbb{N}^2$ such that $|qx-p|< 1/q$. Continued fractions provide a straightforward method for finding these ``good" rational approximations $p/q$.
In other words, we have
\[
\min\limits_{p\in\mathbb{N}, 1\leq q\leq q_{n}(x)}\left|x-\frac{p}{q}\right|=\left|x-\frac{p_n(x)}{q_n(x)}\right|.
\]

Given $t_0\geq1$, let $\Psi:[t_0,\infty)\to\mathbb{R^+}$ be a non-increasing function. Let $\mathcal{D}(\Psi)$ denote the set of $\Psi$-Dirichlet improvable numbers, that is, the set of all $x\in(0,1)$ for which there exists $T>t_0$ such that for every $t>T$, the inequalities
\[|qx-p|<\Psi(t)~~ \text{and}~~ 1\leq q<t,\]
have non-trivial solutions $(p,q)\in\mathbb{N}^2$. Elements of the complementary set, denoted by $\mathcal{D}^{c}(\Psi)$, are called $\Psi$-Dirichlet non-improvable numbers. The study of the metrical properties of $\mathcal{D}(\Psi)$ goes back to the work of Davenport and Schmidt \cite[Theorem 1]{DS69}, who proved that, for any $0<c<1$, the set $\mathcal{D}(c/t)$ is contained in a union of the set of rational numbers and the set of irrational numbers with uniformly bounded partial quotients. 
 In 2008, Kleinbock and Wadleigh \cite[Theorem 1.8]{kleinbock2018zero} established a zero-one law for the Lebesgue measure of $\mathcal{D}^{c}(\Psi)$.  Subsequently, Hussain, Kleinbock, Wadleigh, and Wang \cite{hussain2018hausdorff} established a zero-infinity law for the set $\mathcal{D}^{c}(\Psi)$ in the sense of $g$-dimensional Hausdorff measure, where $g$ is an essentially sub-linear dimension function. Bos, Hussain and Simmons \cite{bos2023generalised} recently generalized the Hausdorff measure of $\mathcal{D}^{c}(\Psi)$ to all dimension functions under natural, non-restrictive conditions. It is worth pointing out that Kleinbock and Wadleigh \cite[Lemma 2.2]{kleinbock2018zero} provided a useful criterion based on continued fractions to determine whether a real number belongs to $\mathcal{D}(\Psi)$ under the condition that $t\Psi(t) < 1$ for all $t \geq t_0$. More precisely, they showed that 
\begin{align*}
  \big\{x\in(0,1)\colon &a_n(x)a_{n+1}(x)\geq \frac{q_n(x)\Psi(q_n(x))}{\big(1-q_n(x)\Psi(q_n(x))\big)}\,\    \text{for infinitely many} \  n\in \mathbb{N}\big\}\subseteq \mathcal{D}^c(\Psi)\\
 & \subseteq\big\{x\in(0,1)\colon a_n(x)a_{n+1}(x)\geq \frac{q_n(x)\Psi(q_n(x))}{4\big(1-q_n(x)\Psi(q_n(x))\big)}\,  \ \text{for infinitely many}\  n\in \mathbb{N}\big\}.
\end{align*}

This demonstrates that the behavior of the product of consecutive partial quotients is crucial in studying the set of Dirichlet non-improvable numbers. Later on, many interests have been drawn to the growth rate of the product of consecutive partial quotients from various perspectives. See Bakhrawar, Hussain, Kleinbock and Wang \cite{BHKW2022}, Fang, Ma, Song and Yang \cite{FMSY24}, Huang, Wu and Xu \cite{huang2020metric}, Bakhrawar and Feng \cite{BF2022},   Hussain and Shulga \cite{hussain2024} for example.

In another direction, inspired by the works of Khinchin \cite{K35} and Diamond and Vaaler \cite{DV86},  Hu, Hussain, and Yu \cite{hu2021limit} investigated metrical properties related to the sum and the maximum of the product of consecutive partial quotients, defined by
$$S_n(x)=\sum_{i=1}^{n}a_i(x)a_{i+1}(x)\ \ \text{and}\ \ L_n(x)=\max_{1\le i\le n}\big\{a_i(x)a_{i+1}(x)\big\}.$$
In particular, they proved that $S_n(x)/(n\log^2n)$ converges to $1/(2\log 2)$ in Lebesgue measure.
  For the strong law of large numbers, a similar approach to that used by Philipp \cite{philipp1988limit} can show that there is no reasonably regular function such that the ratio of the sum $S_n(x)$ to the function converges to a finite non-zero constant for Lebesgue almost all $x\in(0,1)$. However, a result of Hu et al. \cite[Theorem 1.5]{hu2021limit} shows that the maximum $L_n(x)$ is responsible for the failure of the strong law of large numbers. Specifically, for Lebesgue almost all $x \in (0,1)$, 
\[
\lim_{n \to \infty} \frac{S_n(x) - L_n(x)}{n \log^2 n} = \frac{1}{2 \log 2}.
\]

 Hu et al. \cite{hu2021limit} also showed that for Lebesgue almost all $x\in(0,1)$, 
  $$\liminf_{n\to\infty}\frac{L_n(x)\log\log n}{n\log n}=\frac{1}{2\log 2}.$$
Then, it is natural to study the points for which  $L_n(x)$ grows at different rates. 
More precisely,  we are interested in the Hausdorff dimension of the level set 
\begin{equation*}
  L(\varphi):=\left\{x\in (0,1):\lim_{n\to \infty}\frac{L_n(x)}{\varphi(n)}=1\right\},
\end{equation*}
where $\varphi:\mathbb{R^+}\to\mathbb{R}^{+}$ is an increasing function such that $\log \varphi$ is a regularly increasing function with index $\rho$.
Before stating our main results, we shall introduce some classes of functions with different growth rates, representing typical cases of regularly varying functions as described in \cite{N. Bingham1989}.
\begin{definition}
  Let $c>0$ be a constant. A function $f\in C^1\left([c,\infty)\right)$ is said to be a regularly increasing function with index $\rho$ if $f(x)>0$, $\lim_{x\to\infty}f(x)=\infty$, $f'(x)>0$, and
  \begin{equation}\label{def1}
    \lim_{x\to\infty}\frac{xf'(x)}{f(x)}=\rho\in [0,\infty].
  \end{equation}
\end{definition}
The definition and principal properties of regularly increasing functions are due to Karamata \cite{Kara30} in the case of continuous functions, and to Korevaar, van Aardenne-Ehrenfest and de Bruijn \cite{kore49} in the case of measurable functions. Regularly increasing functions frequently arise in number theory and probability theory. Jakimczuk \cite{J2011} employed regularly increasing functions with index 
0 to study the asymptotic behavior of Bell numbers. Chang and Chen \cite[Theorem 1.2]{CC18} determined the Hausdorff dimension of level sets associated with the growth rate of the maximum of partial quotients among the first 
$n$ terms for regularly increasing functions with index 
0. This result was recently extended by Fang and Liu \cite[Theorem 1.8]{fang2021largest}, the authors show that if $\log \varphi$ is a regularly increasing function with index $\rho$, then
\[
\dim_{\mathrm{H}} \left\{x \in (0,1) : \lim_{n \to \infty} \frac{\max_{1 \le i \le n} a_i(x)}{\varphi(n)} = 1 \right\}=
\begin{cases}
1, & \text{if } 0<\rho<1/2; \\
1/2, & \text{if}~1/2<\rho<\infty;\\
\frac{1}{2 + \limsup_{n \to \infty} \frac{\log \varphi(n+1)}{\sum_{k=1}^n \log \varphi(k)}}, & \text{if }\rho=\infty.
\end{cases}
\]

Our results reveal that the Hausdorff dimension of $L(\varphi)$ decreases continuously from $1$ to $0$ in a certain sense, depending on the index $\rho$ of the regularly increasing function $\log\varphi$. In what follows, we use the notation $\dim_{\rm H}$ to denote the Hausdorff dimension. Now, we are in a position to state the main results. 
\begin{theorem}\label{main thm1}
  Let $\varphi:\mathbb{R^+}\to\mathbb{R}^+$ be an increasing function such that $\log \varphi$ is a regularly increasing function with index $\rho$. Then we have 
  \begin{itemize}
    \item[(1)] $\dim_{\mathrm{H}}L(\varphi)=1$, if $0\le \rho<1/2$, 
    \item[(2)] $\dim_{\mathrm{H}}L(\varphi)=1/2$, if $1/2< \rho\le 1$,
    \item[(3)] $\dim_{\mathrm{H}}L(\varphi)=\frac{1}{1+\beta}$, if $1<\rho\le \infty$,
  where $\beta$ is given by
  \begin{equation}\label{beta}
    \beta:=\limsup_{n\to\infty}\frac{\log\varphi(n+1)+\log\varphi(n-1)+\cdots+\log\varphi(n+1-2\lfloor n/2\rfloor)}{\log\varphi(n)+\log\varphi(n-2)+\cdots+\log\varphi(n-2\lfloor(n-1)/2\rfloor)}.  
  \end{equation}
  \end{itemize}
\end{theorem} 
 
Before proceeding, we give some remarks.
\begin{itemize}
\item Using the same method as Fang and Liu \cite[Lemma 3.1]{fang2021largest}, we can establish that $L(\varphi)$ is non-empty if and only if $\varphi$ is equivalent to an increasing function. Thus, we always assume that $\varphi$ is increasing when analyzing $L(\varphi)$.
\item In \cite{J2011}, the function $f$ is referred to as being of slow increase when $\rho=0$. Functions such as $\log x$, $\log\log x$, $(\log x)^a$ with $a\in \mathbb{R}$ and $e^{(\log x)^b}$ with $0<b<1$ regularly increase with the index $\rho =0$; functions such as $e^x$, $xe^x$ and $e^x/x^2$ are regularly increasing with index $\rho=\infty$. 
\item Let $1<b<\infty$ and \(\varphi(x) = e^{x^{b} \log x}\). Then \(\log \varphi\) is regularly increasing with index \(1<\rho<\infty\) and 
\[\dim_{\mathrm{H}}L(\varphi)=\frac{1}{2}.\]
\item Let $1<b<\infty$ and $\log\varphi(x)=b^{x^r}$ with $r>0$. Then $\log\varphi$ is a regularly increasing function with index $\rho=\infty$ and  
\[
\dim_{\mathrm{H}}L(\varphi)=
\begin{cases}
\frac{1}{2}, & \text{if}~ 0<r<1;\\
     \frac{1}{1+b}, & \text{if}~ r=1;\\
    0, & \text{if}~ r>1. \end{cases}\]
  \item In Theorem~\ref{main thm1}~(3), if we choose $\varphi(n) = e^{b^n}$ with $b > 1$, then $\log \varphi$ is regularly increasing with index $\rho = \infty$, and we obtain
\[
\dim_{\mathrm{H}} L(\varphi) = \frac{1}{1 + b}.
\]
In this case, the parameter $\beta$ defined in \eqref{beta} simplifies to the form
\begin{equation}\label{beta2}
    1 + \limsup_{n \to \infty} \frac{\log \varphi(n+1)}{\sum_{k=1}^n \log \varphi(k)},
\end{equation}
which coincides with the result in \cite[Theorem 1.8]{fang2021largest}.
However, if we consider the function
\[
\varphi(x) = \exp\left(x^a \cdot \left(\log 2 + \frac{\log 3 - \log 2}{1 + x^b} \cdot \cos(\pi x)\right)\right)
\]
with $1 < a \le \infty$ and $b > 0$, then $\log \varphi$ is still regularly increasing with index $1 < \rho = a \le \infty$. For sufficiently large positive integers $x$, this function behaves asymptotically as
\[
\varphi(n) = 
\begin{cases}
2^{n^a}, & \text{if } n \text{ is odd}; \\
3^{n^a}, & \text{if } n \text{ is even}.
\end{cases}
\]
As a consequence, the value of the parameter $\beta$ in \eqref{beta} becomes $\log 3 / \log 2$, whereas the value of the simplified form in \eqref{beta2} is $1$. This case shows that the expression for $\beta$ cannot be reduced to the form given in \eqref{beta2}, and the general formulation is therefore necessary.
\end{itemize}

For the critical case $\rho=1/2$, we construct two regularly increasing functions to show that there exists a jump of the Hausdorff dimension of $L(\varphi)$.
\begin{theorem}\label{mainthm3}
  Let  $R:\mathbb{R}^+\to\mathbb{R}^+$ be a regularly increasing function with index $0$. Then we have
    \begin{align*}
    \dim_{\mathrm{H}}L(\varphi)=\begin{cases}
        1,&\text{if}~~\log\varphi(n)=\sqrt{n}/R(n);\\
        1/2,& \text{if}~~\log\varphi(n)=\sqrt{n}R(n),
    \end{cases}
\end{align*}
\end{theorem}

 Thus, we are committed to constructing a discontinuous function $\psi$ such that the Hausdorff dimension of $L(\psi)$ decreases continuously from 1 to 1/2. The following result provides an answer. To explain this, we need to introduce the pressure function $P(\theta)$, defined by
\begin{equation}\label{pressure function}
  P(\theta):=\lim_{n\to\infty}\frac{1}{n}\log \sum_{a_1,\ldots,a_n\in \mathbb{N}}q_n^{-2\theta}(a_1, \ldots, a_n),~\forall\ \theta>\frac{1}{2}.
\end{equation}
The pressure function $\theta\mapsto P(\theta)$ was shown to be strictly decreasing, convex and real analytical in $(1/2,\infty)$, and admits a singularity in $1/2$ (see \cite{KS07}). Our conclusion is
stated in the following theorem.
\begin{theorem}\label{mainthm2}
  Let $c\in (0,\infty)$ and $\log\psi(n)=c\lfloor\sqrt{n}\rfloor$. Then we have
  $$\dim_{\mathrm{H}}L(\psi)=\theta(c),$$
  where $\theta(c)$ is the unique real solution of the equation
  $$P(\theta)=c\left(\theta-\frac{1}{2}\right).$$
\end{theorem}
 
 Before proceeding, we give some remarks.
\begin{itemize}
\item Notice that the function $c\mapsto \theta(c)$ decreases from 1 to 1/2 as $c$ changes from 0 to $\infty$.
 \item  There exist uncountably many discontinuous functions for which the Hausdorff dimension of $L(\psi)$ continuously decreases from 1 to 1/2. Examples include $\psi(n)=R(n)e^{c\lfloor\sqrt{n}\rfloor}$ and $\psi(n)=e^{c\lfloor\sqrt{n}\rfloor}/R(n)$, where $R$ is increasing regularly with the index $\rho=0$.
  \item If we consider the level set of maximal multiple products of consecutive partial quotients in continued fractions, the above results remain unchanged except for Theorem \ref{main thm1} (3). It may require a different approach, and we have not yet established a proof for this case.
 \end{itemize}
\begin{figure}[h]
\centering
        \includegraphics[width=0.4\textwidth]{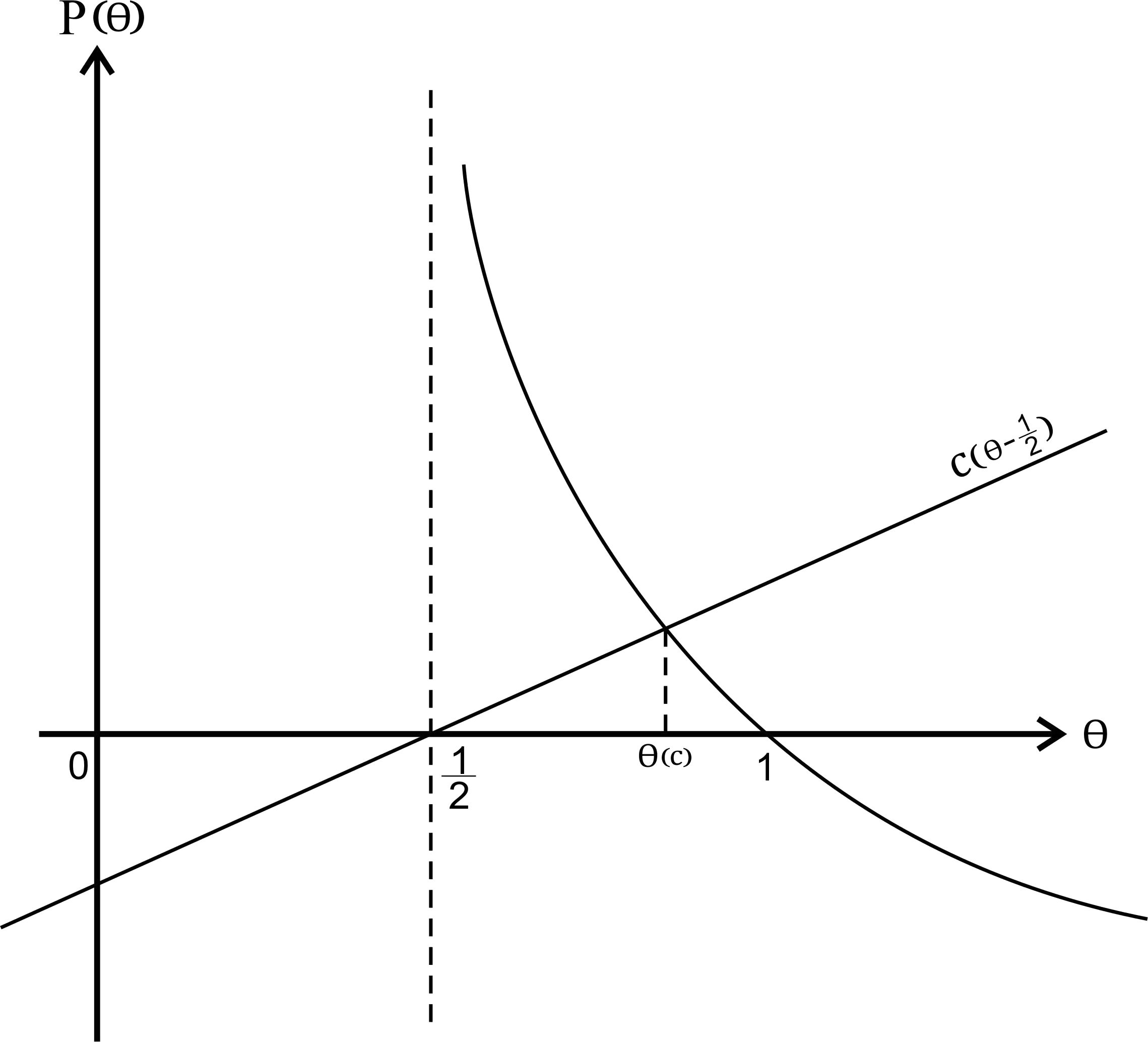}
        \caption{The illustration for the solution of the pressure equation in Theorem \ref{mainthm2}}
\end{figure}

Throughout this paper, we use $|\cdot|$ to denote the length of a subinterval of $(0,1)$, 
$\mathcal{H}^{s}$ to denote the $s$-dimensional Hausdorff measure of a set, 
$\lfloor x\rfloor$ the largest integer not exceeding $x$ and $\sharp$ the cardinality of a set, respectively. The paper is organized as follows. In Section 2, we present some elementary properties and dimensional results on continued fractions and regularly increasing functions. Sections 3, 4 and 5 are devoted to the proofs of the main results.

\section{Preliminaries}
\subsection{Elementary properties of continued fractions}
For any $n\geq1$ and $(a_1,\ldots,a_n)\in \mathbb{N}^n$, we call
$$I_n(a_1,\ldots,a_n)=\{x\in(0,1):a_1(x)=a_1,\ldots,a_n(x)=a_n\}$$
a basic interval of order $n$. By \eqref{jd}, we know that all the points in $I_{n}(a_1, \ldots, a_n)$ have the same $p_n(x)$ and $q_n(x)$. Thus, we write
$p_n(a_1,\ldots,a_n)=p_n(x)=p_n$ and $q_n(a_1,\ldots,a_n)=q_n(x)=q_n$ for $x\in I_{n}(a_1, \ldots, a_n)$.
It is well known (see \cite[p. 4]{Khi64}) that $p_n$ and $q_n$ satisfy the following recursive formula:
\begin{equation}\label{ppqq}
\begin{cases}
p_{-1}=1,\ \ p_0=0,\ \ p_n=a_np_{n-1}+p_{n-2}\ (n\geq1);\cr
q_{-1}=0,\ \ \ q_0=1,\ \ q_n=a_nq_{n-1}+q_{n-2}\ (n\geq1).
\end{cases}
\end{equation}
%As a consequence, we have the following results.
%\begin{lemma}\label{Per(length)}(\cite[Proposition 2.1]{tan2023distribution})
%  For any $n\ge 1$ and $x\in(0,1)$, we have
%  \begin{itemize}
 %   \item [(1)] $q_n\ge 2^{\frac{n-1}{2}}$;
%    \item [(2)] $\prod_{i=1}^n a_i\le q_n\le \prod_{i=1}^n (a_n+1)$;
%    \item [(3)] if $1\le i\le n$, then
%  $$\frac{a_i+1}{2}\le \frac{q_n(a_1,\ldots,a_{i-1},a_i,a_{i+1},\ldots,a_n)}{q_{n-1}(a_1,\ldots,a_{i-1},a_{i+1},\ldots,a_n)}\le a_i+1.$$
 % \end{itemize} 
%\end{lemma}
\begin{proposition}
[{\cite[p. 18]{iosifescu2002metrical}}]\label{cd}
For any $(a_1,\ldots, a_n)\in\mathbb{N}^{n}$, $I_{n}(a_1,\ldots, a_n)$ is the interval with the endpoints
$p_n/q_n$ and $(p_n+p_{n-1})/(q_n+q_{n-1})$.  More precisely,
\begin{equation*}
I_{n}(a_1,\ldots, a_n)=
\begin{cases}
\big[\frac{p_n}{q_n},\frac{p_n+p_{n-1}}{q_n+q_{n-1}}\big),\ \ \ \text{if}\ n\ \text{is even},\cr
\big(\frac{p_n+p_{n-1}}{q_n+q_{n-1}},\frac{p_n}{q_n}\big],\ \ \ \text{if}\ n\ \text{is odd}.
\end{cases}
\end{equation*}
As a result, the length of $I_{n}(a_1, \ldots, a_n)$ equals to
\begin{equation*}
|I_{n}(a_1, \ldots, a_n)|=\frac{1}{q_n(q_n+q_{n-1})}.
\end{equation*}
\end{proposition}
Combining the second of formula \eqref{ppqq} and Proposition \ref{cd}, we deduce that 
\begin{equation}\label{lengthIn}
  2^{-(2n+1)}\prod_{k=1}^{n}a_k^{-2}\le |I_n(a_1,\ldots,a_n)|\le \prod_{k=1}^{n}a_k^{-2}.
\end{equation}

The following result can be viewed as the bounded distortion property of continued fractions.
\begin{lemma}(\cite[Lemma A.2]{Mor18})\label{lemma2.2}
  For any $(a_1, \ldots, a_n)\in \mathbb{N}^n$ and $(b_1, \ldots, b_k)\in \mathbb{N}^k$,
  $$\frac{1}{2}\le |\frac{I_{n+k}(a_1,\ldots,a_n,b_1,\ldots,b_k)|}{|I_n(a_1, \ldots, a_n)|\cdot |I_k(b_1, \ldots, b_k)|}\le 2.$$
  As a consequence, for any $(a_1,\ldots, a_n, \ldots, a_{n+k})\in \mathbb{N}^{n+k}$,
  $$\frac{1}{8}\le \frac{|I_{n+k}(a_1,\ldots, a_n, \ldots, a_{n+k})|}{|I_1(a_n)|\cdot |I_{n+k-1}(a_1,\ldots,a_{n-1}, a_{n+1}, \ldots, a_{n+k})|}\le 8.$$
\end{lemma}

\subsection{Some useful lemmas for estimating Hausdorff dimension}
Let $\{n_k\}$ be a strictly increasing sequence of positive integers, and $\{s_k\}$ and $\{t_k\}$ be two sequences of positive numbers with $s_k, t_k\to \infty$ as $k\to \infty$. For any $M\in\mathbb{N}$, define
\begin{align}\label{nst}
  \nonumber E\left(\{n_k\}, \{s_k\}, \{t_k\}\right):= \Big\{x\in (0,1):& ~s_k<a_{n_k}\le s_k+t_k~\text{for all large}~k\in \mathbb{N}, \\
   &~1\leq a_j(x)\leq M~\text{for all}~j\neq n_k\Big\}.
\end{align}
For the sequences $\{n_k\}, \{s_k\}$ and $\{t_k\}$, we make the following assumptions:
\begin{itemize}
  \item[\textbf{(H1)}] $\{n_k\}$ satisfies that $n_k/k\to \infty$ as $k\to\infty$;
  \item[\textbf{(H2)}] $\{s_k\}$ and $\{t_k\}$ are logarithmically equivalent in the sense that
      $$\lim_{k\to\infty}\frac{\log s_k}{\log t_k}=1;$$
  \item[\textbf{(H3)}] there exist two real numbers $\alpha_1,\alpha_2$ such that 
  $$\alpha_1:=\lim_{k\to\infty}\frac{1}{n_k}\sum_{j=1}^{k} \log s_j\quad\text{and}\quad\alpha_2:=\lim_{k\to\infty} \frac{1}{n_k}\log s_k.$$
\end{itemize}
\begin{lemma}(\cite[Theorem A]{Fangkey})\label{fangkey}
  Under hypotheses \textbf{(H1)}, \textbf{(H2)} and \textbf{(H3)}, we have
  \begin{itemize}
      \item[(1)] when $\alpha_1=0$,
      $$\dim_{\mathrm{H}} E\left(\{n_k\}, \{s_k\}, \{t_k\}\right)=1,$$
      \item[(2)] when $\alpha_1\in (0,\infty)$,
  $$\dim_{\mathrm{H}} E\left(\{n_k\}, \{s_k\}, \{t_k\}\right)=\theta_1(\alpha_1, \alpha_2),$$%=\dim_{\mathrm{H}} E_L\left(\{n_k\}, \{s_k\}, \{t_k\}\right)
  where $\theta_1(\alpha_1, \alpha_2)$ is the unique real solution of the pressure equation
  $$P(\theta)=(2\alpha_1-\alpha_2)\theta-(\alpha_1-\alpha_2).$$ 
  \end{itemize}
  
\end{lemma}

The following dimensional result is useful for obtaining the lower bound estimates of the Hausdorff dimension of sets in continued fractions. 
Let $\{s_n\}$ and $\{t_n\}$ be two sequences of positive real numbers. Assume that $s_n$, $t_n\to\infty$ as $n\to\infty$, $s_n>t_n$ and
$\liminf\limits_{n\to\infty}\frac{s_n-t_n}{s_n}>0$.
For any $N\geq1$, let
\begin{equation}\label{Bst}
  B(\{s_n\}, \{t_n\}, N):= \{x\in (0,1): s_n-t_n<a_n(x) \le s_n+t_n, \forall n\ge N\}.
\end{equation}
\begin{lemma}(\cite[Lemma 2.3]{LiaoRams2019})\label{LemmaLR}
  For any $N\geq1$, we have 
  $$\dim_{\mathrm{H}}B(\{s_n\}, \{t_n\}, N)=\liminf_{n\to\infty}\frac{\sum_{k=1}^{n}\log t_k}{2\sum_{k=1}^{n+1}\log s_k-\log t_{n+1}}.$$
\end{lemma}
It is worth noting that
\[
\dim_{\rm H}B(\{s_{n}\},\{t_{n}\},N)=\dim_{\rm H}B(\{s_{n}\},\{t_{n}\},1).
\]
Indeed, this equality follows from the fact that the dimensional formula in Lemma \ref{LemmaLR} is not affected by a finite number of initial terms in the sequences $\{s_n\}$ and $\{t_n\}$. Moreover, the set $B(\{s_{n}\},\{t_{n}\}, N)$ can be expressed as a countable union of bi-Lipschitz images of $B(\{s_{n+N-1}\},\{t_{n+N-1}\},1)$. Since bi-Lipschitz maps preserve the Hausdorff dimension, the equality holds.

The construction of Cantor-type subsets is another effective method to obtain lower bounds for the Hausdorff dimension of fractal sets. A Cantor-type set
is defined as follows. Let $[0,1]=E_0\supseteq E_1\supseteq E_2\supseteq\ldots$ be a decreasing sequence of sets such that each $E_n$ is a union of finite number of disjoint closed intervals, with each interval of $E_n$ containing at least two intervals of $E_{n+1}$, and the maximum length of intervals in $E_n$ tending to $0$ as $n$ tending to infinity. Then the set
\begin{equation}\label{cse}
E:=\bigcap_{n\geq0}E_n
\end{equation}
is a totally disconnected subset of $[0,1]$. The following lemma provides a lower bound of $\dim_{\rm H} E$.
\begin{lemma}(\cite[Example 4.6]{falconer2004fractal})\label{ex}
  Suppose that for any positive integer $n\geq1$, each interval of $E_{n-1}$ contains at least $m_n$ intervals of $E_n$ which are separated by
gaps of at least $\theta_n$ in the general construction \eqref{cse}. If $m_n\geq2$ and $0<\theta_{n+1}<\theta_n$, then we have 
  $$\dim_{\mathrm{H}}E\ge \liminf_{n\to\infty}\frac{\log(m_1\cdots m_{n-1})}{-\log(m_n\theta_n)}.$$
\end{lemma}
To end this subsection, we present a method to estimate an upper bound of the Hausdorff dimension.
\begin{lemma}(\cite[Proposition 4.1]{falconer2004fractal})\label{sjgh}
  Suppose a set $F$ can be covered by $n_k$ sets of diameter as most $\delta_k\to 0$ as $k\to \infty$. Then
  $$\dim_{\mathrm{H}}F\le \liminf_{k\to \infty}\frac{\log n_k}{-\log\delta_k}.$$
\end{lemma}

\subsection{Regularly increasing functions }
In this subsection, we collect some basic properties of regularly increasing functions.
\begin{lemma}[\cite{V. M}]\label{i1}
Let $f$ be a regularly increasing function with index $\rho$.
 \begin{itemize}
\item[(1)] If $\rho=0$, then for any $\varepsilon>0$, $x^{-\varepsilon}f(x)$ is ultimately decreasing and $x^{\varepsilon}f(x)$ is ultimately increasing. Moreover, we have 
\[\lim_{x\to\infty}x^{-\varepsilon}f(x)=0\ \ 
\text{and}\ \ \lim_{x\to\infty}x^{\varepsilon}f(x)=\infty.\]
\item[(2)] If $\rho\in(0,\infty)$, then for $\alpha>0$ and $C>0$, $f(x^{\alpha}/C)$ is regularly increasing with index $\alpha\rho$.
\item[(3)] If $\rho\in [0,\infty)$, then for $C\in \mathbb{R}$, $\lim\limits_{x\to\infty}f(x+C)/f(x)=1$.
\item[(4)] If $\rho\in[0,\infty]$, then
      $\lim\limits_{x\to\infty}\log f(x)/\log x=\rho$.
    \end{itemize}
  \end{lemma}
  
\begin{lemma}[\cite{fang2021largest}]\label{i2}
Let $f$ be a regularly increasing function with index $\rho$. 
  \begin{itemize}
    \item[(1)]If $\rho\in(0,1)$, then
    $\lim\limits_{x\to\infty}(f(x+1)-f(x))=0$.
    \item[(2)]If $\rho\in(1,\infty]$, then
    $\lim\limits_{x\to\infty}(f(x+1)-f(x))=\infty$.
    \item[(3)] If $\rho=0$, then for any $\alpha>0$, letting $g(x)=x/f(x^{\alpha})$ and $h(x)=xf(x^{\alpha})$, we have $f$ and $g$ are ultimately increasing. Moreover,
    \[\lim_{x\to\infty}(g(x+1)-g(x))=0\ \ \text{and}\ \ 
\lim_{x\to\infty}(h(x+1)-h(x))=\infty.\]
  \end{itemize}
\end{lemma}

\section{Proof of Theorem \ref{main thm1}}
This section gives a proof of Theorem \ref{main thm1}. Our idea is mainly inspired by Fang, Moreira and Zhang \cite{Fangkey}, Hu, Hussain, and Yu \cite{hu2021limit} and Liao, Rams \cite{LiaoRams2}. We divide the proof into three cases. Recall that 
\[L(\varphi)=\left\{x\in (0,1):\lim_{n\to \infty}\frac{L_n(x)}{\varphi(n)}=1\right\},\] where $\varphi:\mathbb{R^+}\to\mathbb{R^+}$ is an increasing function such that $\log \varphi$ is a regularly increasing  with index $\rho$.
\subsection{\textbf{The case $0\leq\rho<1/2$}.}
 Using the properties of the regularly increasing function $\log\varphi$ with index $\rho$, our strategy is to construct a suitable Cantor-type subset $E(\{n_k\},\{s_k\},\{t_k\})$, defined as \eqref{nst}, of $L(\varphi)$.
 
  Since $\log\varphi$ is regularly increasing with index $\rho$, satisfying $0\le \rho< 1/2$. By Lemma \ref{i1} (4), we have
 \begin{equation}\label{jx11}
   \lim_{n\to\infty}\frac{\log\log \varphi(n)}{\log n}=\rho.
 \end{equation}
For any $0<\delta<\frac{1}{2}-\rho$, it follows from \eqref{jx11} that $\log\varphi(n)\le n^{\rho+\delta}$ for sufficiently large $n$. Thus,
\begin{equation}\label{ghn}
\limsup_{n\to\infty}\frac{\log\varphi(n)}{\sqrt{n}}=0.
\end{equation}
For $k\ge 1$, let $n_k=k^2$, $s_k=\lfloor\varphi(n_k)\rfloor$ and $t_k=\lfloor\varphi(n_k)\rfloor/k$. Define 
\begin{align}\label{subset1}
  E(\{n_k\},\{s_k\},\{t_k\})=\big\{x\in (0,1):\ &\nonumber \lfloor\varphi(n_k)\rfloor\le a_{n_k}(x)\le(1+\frac{1}{k})\lfloor\varphi(n_k)\rfloor,~\text{for all}~k\ge 1, \\
  &a_i(x)=1\ \text{for any}\ i\neq n_k\big\}.
\end{align}
\begin{proposition}\label{bh2}
  Let $E(\{n_k\},\{s_k\},\{t_k\})$ be defined as in (\ref{subset1}). Then we have 
  $$E(\{n_k\},\{s_k\},\{t_k\})\subseteq L(\varphi).$$
\end{proposition}
\begin{proof}
Consider the function $\log\varphi\left(x^2\right)$. From Lemma \ref{i1} (2), we deduce that $\log\varphi\left(x^2\right)$ is regularly increasing with index $2\rho<1$. Then, by Lemma \ref{i2} (1), we have
\[
\lim_{k\to\infty}\left(\log\varphi(n_{k+1})-\log\varphi(n_k)\right)=0~~ \text{and}~~
\lim_{k\to\infty}\frac{\varphi(n_{k+1})}{\varphi(n_k)}=1.
\]
Let $x\in E(\{n_k\},\{s_k\},\{t_k\})$ be fixed, we have
\[\lfloor\varphi(n_k)\rfloor\le a_{n_k}(x)a_{n_k+1}(x)\le \left(1+\frac{1}{k}\right)\lfloor\varphi(n_k)\rfloor.\]
For any sufficiently large $n$, there exists a positive integer $k$ such that $n_k\le n< n_{k+1}$. Since $\varphi$ is increasing, we have $\varphi(n_k)\le \varphi(n)< \varphi(n_{k+1})$
  and
  $$L_n(x)=\max\{a_1(x)a_2(x),\ldots,a_n(x)a_{n+1}(x)\}=a_{n_k}(x)a_{n_k+1}(x).$$
   Therefore, we obtain
  $$1=\liminf_{k\to \infty}\frac{\varphi(m_k)-1}{\varphi(m_{k+1})}
  \le \liminf_{n\to \infty}\frac{L_n(x)}{\varphi(n)}
  \le \limsup_{n\to \infty}\frac{L_n(x)}{\varphi(n)}
  \le \limsup_{k\to \infty}\frac{(1+1/k)\lfloor\varphi(m_k)\rfloor}{\varphi(m_k)}=1.$$
\end{proof}

By the definition of $n_k$, we deduce from \eqref{ghn} that
\begin{align*}
    \alpha_1=&\lim_{k\to\infty}\frac{1}{n_k}\sum_{j=1}^k\log s_j=\lim_{k\to\infty}\frac{1}{k^2}\sum_{j=1}^k\log\lfloor\varphi(j^2)\rfloor\\
    &\le \limsup_{n\to\infty}\frac{1}{k^2}\left( \log\varphi\left(1^2\right)+ \log\varphi\left(2^2\right)+\cdots+  \log\varphi\left(k^2\right)\right)\\
&\le\limsup_{n\to\infty} \frac{\log \varphi\left(k^2\right)}{k}=0.
\end{align*}
Hence, by using Lemma \ref{fangkey} (1), we deduce that
\[\dim_{\mathrm{H}}L(\varphi)\ge \dim_{\mathrm{H}}E(\{n_k\},\{s_k\},\{t_k\})=1.\]

\subsection{\textbf{The case $1/2<\rho\le 1$}.} 
In this subsection, we divide the proof into three parts: the lower bounds for the case $1/2<\rho<1$ and $\rho=1$, the upper bound for the case $1/2<\rho\le 1$.

\subsubsection{Lower bound for the case $1/2<\rho<1$.}
To establish a lower bound for $\dim_{\rm{H}}L(\varphi)$ in the case $1/2<\rho<1$, we construct a Cantor-type subset of $L(\varphi)$ using Lemma \ref{LemmaLR}.
\begin{proposition}
For any $n\geq1$, let $s_n=\sqrt{\varphi(n)}$ and $t_n=\sqrt{\varphi(n)/n}$. Then, we have
  $$B\left(\{s_n\}, \{t_n\}, 1\right)\subseteq L(\varphi),$$
  where $B(\{s_n\}, \{t_n\}, 1)$ is defined as in (\ref{Bst}).  
\end{proposition}
\begin{proof}
First, we show that $\varphi(n)/n$ tends to infinity as $n \to\infty$. Consider the function $f(x)=\varphi(x)/x$. Since $\log\varphi$ is regularly increasing with index $1/2<\rho<1$, then by \eqref{def1}, we have
$$f'(x)=\frac{x\varphi'(x)-\varphi(x)}{x^2} >\frac{\varphi(x)\left(\frac{\log\varphi(x)}{2}-1\right)}{x^2}>0,$$
for sufficiently large $x$. Hence, there exists a positive integer $N_1\geq2$ such that $n \mapsto \varphi(n)/n$, where $n\geq N_1$, is increasing and tends to infinity as $n\to\infty$.
For any $x\in B(\{s_n\}, \{t_n\}, 1)$, by the definition of $L_n(x)$, we have  
\begin{equation}\label{Lnx}
\begin{cases}
L_n(x)\ge a_n(x)a_{n+1}(x)
\geq\left(1-1/\sqrt{n}\right)^2\varphi(n),\cr
L_n(x)\leq\max\limits_{1\le k\le n}\{(s_k+t_k)(s_{k+1}+t_{k+1})\}\leq4\varphi(N-1)+\left(1+1/\sqrt{n+1}\right)^2\varphi(n+1).
\end{cases}
\end{equation}
By Lemma \ref{i2} (1), we obtain that 
\begin{equation}\label{phi}
  \lim_{n\to\infty}\frac{\varphi(n+1)}{\varphi(n)}=1.
\end{equation}
It follows from \eqref{Lnx} and \eqref{phi} that
\[
    1\le\liminf_{n\to \infty}\frac{L_n(x)}{\varphi(n)}  \le\limsup_{n\to \infty}\frac{L_n(x)}{\varphi(n)}\leq\limsup_{n\to \infty}\frac{\varphi(n+1)}{\varphi(n)}=1,
\]
  which implies that $x\in L(\varphi)$. Thus, $B(\{s_n\}, \{t_n\}, 1)\subseteq L(\varphi)$.
\end{proof}
Notice that $\log\varphi$ is regularly increasing with index $\rho> 1/2$. By Lemma \ref{i1} (3) and (4), we have
\begin{equation}\label{phibz}
  \lim_{n\to\infty}\frac{\log\varphi(n+1)}{\log\varphi(n)}=1\ \text{and}\ \lim_{n\to\infty}\frac{\log\varphi(n)}{\log n} =\infty.
\end{equation}
 Together with  Lemma \ref{LemmaLR} and \eqref{phibz}, we obtain
 $$\dim_{\mathrm{H}}L(\varphi) \ge\liminf_{n\to\infty}\frac{\frac{1}{2} \left(\sum_{k=1}^{n}\log\varphi(k)-\sum_{k=1}^{n}\log k\right)}{\sum_{k=1}^{n}\log\varphi(k)+ \frac{1}{2}\log\varphi(n+1)+\frac{1}{2}\log(n+1)}=\frac{1}{2}.$$

\subsubsection{Lower bound for the case $\rho=1$.}
If $\rho=1$, the equality \eqref{phi} cannot hold, so we need to provide a new method to obtain the lower bound. To establish a lower bound for $\dim_{\rm{H}}L(\varphi)$ in the case $\rho=1$, we construct a Cantor-type subset of $L(\varphi)$ by using the definition of the regularly increasing function.

Since $\log \varphi$ is a regularly increasing function with index 1. By \eqref{def1}, we have
$\varphi(n)=e^{\alpha n+o(n)}$, where $0<\alpha<\infty$ is a constant. For convenience, we set $\varphi(n)=e^{\alpha n}$, which does not affect the conclusion. 
\begin{enumerate}
    \item[\textbf{Step 1:}] Construct a subset $\Upsilon(\alpha, N)$ of $L(\varphi)$. Let $N\in \mathbb{N}$, define
\begin{equation}\label{upsi}
    \Upsilon(\alpha, N)=\left\{x\in (0,1): e^{-\frac{\alpha}{4}}<\frac{a_n(x)}{e^{\frac{\alpha}{2}n}}< e^{-\frac{\alpha}{4}}+\frac{1}{n},~~\forall n\ge N\right\}.
\end{equation}
Let $N_2$ denote the smallest integer $n$ such that $e^{\frac{\alpha}{2}n}/n\ge 2$. When $N\ge N_2$, the set $\Upsilon(\alpha, N)$ is non-empty.
 For any $x\in \Upsilon(\alpha, N)$. Let $n\ge N$ be sufficiently large, we have $a_{n-1}(x)<a_n(x)$ and
     $$e^{\alpha n}<a_n(x)a_{n+1}(x)<e^{\alpha n}+e^{\alpha n+\alpha/4}\left(\frac{1}{n}+\frac{1}{n+1}\right)+\frac{e^{\alpha n+\alpha/2}}{n(n+1)}.$$
     Thus, for sufficiently large $n\ge N$, we have
     $$L_n(x)=\max_{1\le i\le n}\{a_i(x)a_{i+1}(x)\}=a_n(x)a_{n+1}(x).$$
     It follows that 
     $$\lim_{n\to\infty}\frac{L_n(x)}{\varphi(n)}=1.$$
     Hence,
     $$\Upsilon(\alpha, N)\subseteq L(\varphi).$$
     \item[\textbf{Step 2:}] Construct a measure $\mu$ supported on $\Upsilon(\alpha, N)$. Without loss of generality, we assume $N_2=1$ and set $N=1$.  Then, the number of basic intervals $I_n(a_1,\ldots, a_n)$, which have nonempty intersection with $\Upsilon(\alpha, 1)$, is approximately
     \begin{equation}\label{bjc}
    \prod_{j=1}^n\Big(\frac{1}{j}e^{\frac{\alpha}{2}j}\Big)=\frac{1}{n!}e^{\frac{\alpha}{2}\sum_{j=1}^nj}.
    \end{equation}
    By \eqref{lengthIn}, the length of such interval is
    \begin{equation}\label{bjc2}
    2^{-(2n+1)}\prod_{j=1}^{n}\Big(e^{-\frac{\alpha}{4}}+\frac{1}{j}\Big)^{-2}e^{-\alpha\sum_{j=1}^{n}j}\leq|I_n(a_1,\ldots,a_n)|\leq e^{\frac{n\alpha}{2}-\alpha\sum_{j=1}^{n}j}.
    \end{equation}
    Now, we construct a probability measure $\mu$ uniformly distributed on $\Upsilon(\alpha, 1)$. If  $a_1,\ldots, a_{n-1}$ are given, then the probability of $a_n$ taking any integer value between $e^{\frac{\alpha}{2}n-\frac{\alpha}{4}}$ and $(e^{-\alpha/4}+\frac{1}{n})e^{\frac{\alpha}{2}n}$ is same.  Let $\varepsilon$ be sufficiently small, up to a factor $e^{\varepsilon\sum_{j=1}^{n}j}$, by \eqref{bjc} and \eqref{bjc2}, we have the following relations:
    \begin{enumerate}
       \item[(1)] For the basic intervals $I_n(a_1,\ldots,a_n)$, the length and the measure are given by
       $$|I_n(a_1,\ldots, a_n)|\approx e^{-\alpha\sum_{j=1}^nj}~~\text{and}~~\mu(I_n(a_1,\ldots, a_n))\approx e^{-\frac{\alpha}{2}\sum_{j=1}^n j}.$$
        \item[(2)] All $I_n(a_1,\ldots, a_n)$ contained within a single $I_n(a_1,\ldots, a_{n-1})$ form an interval of length 
        $$e^{\frac{\alpha}{2}n-\alpha\sum_{j=1}^nj}.$$ 
   \end{enumerate}
 
   \item[\textbf{Step 3:}] Estimate the lower bound of $L(\varphi)$.
   For any $x\in \Upsilon(\alpha, 1)$ and $r\in\left(e^{-\alpha\sum_{j=1}^nj}, e^{-\alpha\sum_{j=1}^{n-1}j}\right)$, the measure of the ball $B(x,r)$ is
        \begin{equation*}
            \mu(B(x,r))\approx
            \begin{cases}
                r\cdot e^{-\frac{\alpha}{2}\sum_{j=1}^n j}, &\text{if}~ r< e^{\frac{\alpha}{2}n-\alpha\sum_{j=1}^nj},\\
                e^{-\frac{\alpha}{2}\sum_{j=1}^{n-1} j}, &\text{if}~ r\geq e^{\frac{\alpha}{2}n-\alpha\sum_{j=1}^nj}.
            \end{cases}
        \end{equation*}
        Then, we obtain 
        $$\liminf_{r\to 0}\frac{\log \mu(B(x,r))}{\log r} \ge \liminf_{n\to \infty}\frac{-\frac{\alpha}{2}\sum_{j=1}^{n-1} j}{\frac{\alpha}{2}n-\alpha\sum_{j=1}^n j}=\frac{1}{2}.$$
    Hence, the lower local dimension of $\mu$ equals 1/2 at each point of $\Upsilon(\alpha,1)$, which implies that
    $$\dim_{\mathrm{H}}\Upsilon(\alpha,1)\ge \frac{1}{2}.$$
     By the Frostman Lemma (\cite[Principle 4.2]{falconer2004fractal}), we have
$$\dim_{\mathrm{H}} L(\varphi)\ge \dim_{\mathrm{H}}\Upsilon(\alpha, 1)=\frac{1}{2}.$$
\end{enumerate}

\subsubsection{Upper Bound for the case $1/2<\rho\le 1$.}
To obtain the upper bound of $\dim_{\rm_{H}}L(\varphi)$,  we employ a method of selecting an appropriate positive real number $s$ such that $\mathcal{H}^{s}(L(\varphi))<\infty$. Before proceeding with the proof, we present several key lemmas by choosing this positive real number $s$.
Let
$$\Lambda(m,n):=\big\{(i_1,\ldots,i_n)\in \{1,\ldots,m\}^n:i_1+\cdots+i_n=m\big\},$$
and $\xi(\cdot)$ be the Riemann zeta function.

\begin{lemma}(\cite[Lemma 2.1]{LiaoRams2})\label{LiaoRams2}
  For any $s\in(1/2,1)$ and $m\geq n\geq1$, we have
  $$\sum_{(i_1,\ldots,i_n)\in \Lambda(m,n)}\prod_{k=1}^ni_k^{-2s}\le \left(\frac{9}{2}\left(2+\xi(2s)\right)\right)^nm^{-2s}.$$
\end{lemma}
\begin{lemma}(\cite[Lemma 5.4]{hu2021limit})\label{yyl54}
   For any $\varepsilon>0$ and $n\ge 2$, let
  $$\pi(n)=\#\left\{(a,b)\in\mathbb{N}\times \mathbb{N}:ab=n\right\}.$$
 Then, there exists a constant $c_{\varepsilon}$ depending on $\varepsilon$ such that
  $\pi(n)\le c_{\varepsilon}n^{\varepsilon}$.

\end{lemma}
Let $0<\delta<\rho-1/2$ be fixed. For any $k\ge 1$, define
\begin{equation}\label{eq311}
    \gamma:=\rho-\delta~~\text{and} ~~ n_k:=\lfloor k^{1/\gamma}\rfloor.
\end{equation}
 
\begin{proposition}\label{Lemma3.4}
  Let $n_k$ be defined as in \eqref{eq311}. We have
  $$\liminf_{k\to\infty}\frac{\varphi(n_k)}{\varphi(n_{k-1})}>1.$$
 \end{proposition}
 \begin{proof}
   Since $\varphi$ is a differentiable function, we have
   $$\frac{\varphi(n_k)}{\varphi(n_{k-1})} =\frac{\varphi(n_{k-1})+\int_{n_{k-1}}^{n_k}\varphi'(t)dt}{\varphi(n_{k-1})}.
   $$
  Therefore, it suffices to prove
  $$\liminf_{k\to\infty} \frac{\int_{n_{k-1}}^{n_k}\varphi'(t)dt}{\varphi(n_{k-1})}>0.$$ 
  Note that $\log\varphi$ is an increasing function with index $\rho$, satisfying $1/2<\rho\le 1$. Let $0<\varepsilon<\delta$. By \eqref{def1} and Lemma \ref{i1} (4), we have
  $$\frac{\varphi'(x)}{\varphi(x)}\ge (\rho-\varepsilon)\frac{\log\varphi(x)}{x}\ \ \text{and}\ \ \log\varphi(x)\ge x^{\rho-\varepsilon},$$
  for sufficiently large $x$. Then, for sufficiently large $k$, it follows that 
  \begin{align*}
    \frac{\int_{n_{k-1}}^{n_k}\varphi'(t)dt}{\varphi(n_{k-1})}\ge \int_{n_{k-1}}^{n_k}\frac{\varphi'(t)}{\varphi(t)}dt \ge (\rho-\varepsilon)\int_{n_{k-1}}^{n_k}\frac{\log\varphi(t)}{t}dt\ge (\rho-\varepsilon)\int_{n_{k-1}}^{n_k}\frac{t^{\rho-\varepsilon}}{t}dt= \left(n_k^{\rho-\varepsilon} -n_{k-1}^{\rho-\varepsilon}\right),
  \end{align*}
which, together with $\rho-\varepsilon>\rho-\delta=\gamma$, implies that 
$$\liminf_{k\to\infty} \frac{\int_{n_{k-1}}^{n_k}\varphi'(t)dt}{\varphi(n_{k-1})}>0.$$
The proof is complete.  
 \end{proof}

The next proposition shows the position of the maximal product of consecutive partial quotients among the first $n_k$ terms in the continued fraction expansion of $x$.

\begin{proposition}\label{TandS}
  Let $x\in L(\varphi)$ be fixed. Then, for sufficiently large $k\in \mathbb{N}$, there exists $j_k\ge 1$ such that $n_{k-1}<j_k\leq n_k$ and $L_{n_k}(x)=a_{j_k}(x)a_{j_k+1}(x)$.
\end{proposition}

\begin{proof}
  Let $x\in L(\varphi)$ be fixed. Suppose that there exist infinitely many integers $k_i$ and $j_{k_i}$ with $k_i>k_{i-1}$ such that $j_{k_i}< n_{k_i-1}$ and $L_{n_{k_i}}(x)=a_{j_{k_i}}(x)a_{j_{k_i}+1}(x)$. Then
  $$L_{n_{k_i-1}}(x)=L_{n_{k_i}}(x)=a_{j_{k_i}}(x)a_{j_{k_i}+1}(x).$$
Since $L_n(x)/\varphi(n)\to 1$ as $n\to\infty$, we deduce from Proposition \ref{Lemma3.4} that
$$1=\liminf_{i\to\infty}\frac{L_{n_{k_i-1}}(x)}{\varphi(n_{k_i-1})}
=\liminf_{i\to\infty}\frac{L_{n_{k_i}}(x)}{\varphi(n_{k_i})} \cdot\frac{\varphi(n_{k_i})}{\varphi(n_{k_i-1})} =\liminf_{i\to\infty}\frac{\varphi(n_{k_i})}{\varphi(n_{k_i-1})}>1,
$$
which is a contradiction. Thus, the proof is complete.
\end{proof}

In the following, we construct a cover of the set $L(\varphi)$. Let $s\in(1/2,1)$ be arbitrary. Then for any $x\in L(\varphi)$ and $0<\varepsilon<2s-1$, we have 
$$(1-\varepsilon)\varphi(n)\le L_n(x)\le (1+\varepsilon)\varphi(n),$$
for sufficiently large $n$. Recall that $S_n(x)=\sum_{i=1}^{n}a_i(x)a_{i+1}(x)$.  
By Proposition \ref{TandS}, we obtain that
\begin{equation}\label{eq3.20'}
    (1-\varepsilon)\varphi(n_k)\le L_{n_k}(x)\le S_{n_k}(x)-S_{n_{k-1}}(x)\le S_{n_k}(x)\le n_k L_{n_k}(x)\le (1+\varepsilon)n_k\varphi(n_k),
\end{equation}
for sufficiently large $k$. From Lemma \ref{i1} (4),
We deduce that for sufficiently large $k$, 
\begin{equation}\label{eq3.21'}
    \log\varphi(n_{k})=\log\varphi(\lfloor k^{1/\gamma}\rfloor)\ge \log\varphi(k^{1/\gamma}-1)\ge \log\varphi(k^{1/\gamma}/2)>2k,
\end{equation}
where the last inequality follows from the fact that the function $x\mapsto \log\varphi(x^{1/\gamma}/2)$ is regularly increasing with index $\rho/\gamma>1$.
For sufficiently large $k$, we also have
\begin{equation}\label{eq3.22'}
    n_{k}-n_{k-1} = \lfloor k^{1/\gamma}\rfloor-\lfloor(k-1)^{1/\gamma}\rfloor \le k^{1/\gamma}-(k-1)^{1/\gamma}+1 \le \gamma^{-1}\cdot k^{1/\gamma-1}+1 \le \frac{2}{\gamma}k^{1/\gamma-1}.
\end{equation}
By the choice of $\delta$, we have $1/\gamma-1<1$. Let $K\ge 1$ be an integer such that \eqref{eq3.20'}, \eqref{eq3.21'} and \eqref{eq3.22'} hold for all $k\ge K$. 

For any $k\ge K$, set
$$M_k=\big\{i\in \mathbb{N}:(1-\varepsilon)\varphi(n_k)\le i\le (1+\varepsilon)n_k\varphi(n_k)\big\}.$$
For any $K_1\ge K$ and $k\ge K_1$, define
$$J(\varphi,k, K_1)= \Big\{I_{n_k+1}(a_1,\ldots,a_{n_k+1}): \sum_{j=n_{\ell-1}+1}^{n_{\ell}}a_ja_{j+1}=m_{\ell} ~\text{with}~m_{\ell}\in M_{\ell},~K_1\le \ell \le k\Big\},$$
and
\begin{equation}\label{jk2}
J(\varphi, K_1)=\bigcap_{k=K_1}^{\infty}J(\varphi,k, K_1).
\end{equation}
It follows that 
$$L(\varphi)\subseteq\bigcup_{K_1=K}^{\infty}J(\varphi,K_1).$$

Now, we estimate the upper bound for the Hausdorff dimension of $J(\varphi, K)$. For any other $K_1>K$, we apply the same method to obtain the upper bound for $J(\varphi, K_1)$. For $k\ge K$, the cylinders from $J(\varphi,k, K)$ forms a cover of $J(\varphi, K)$. For any $\ell\ge K $, denote
$$A_{\ell}=\Big\{(a_{n_{\ell-1}+1}, \ldots, a_{n_{\ell}+1})\in \mathbb{N}^{n_{\ell}-n_{\ell-1}+1}:\sum_{j=n_{\ell-1}+1}^{n_{\ell}}a_ja_{j+1}=m_{\ell} ~\text{with}~m_{\ell}\in M_{\ell}\Big\}.$$
Then we have
\begin{align}\label{jk3}
  \nonumber\sum_{I_{n_k+1}\subseteq J(\varphi,k,K)}|I_{n_k+1}|^s &
   \le\sum_{I_{n_k+1}\subseteq J(\varphi,k,K)}\prod_{\ell=K}^{k} (a_{n_{\ell-1}+1}a_{n_{\ell-1}+2}\ldots a_{n_{\ell}}a_{n_{\ell}+1})^{-2s}\\
  \nonumber & \le  \prod_{\ell=K}^{k}\sum_{(a_{n_{l-1}+1}, \ldots, a_{n_{\ell}}a_{n_{\ell}+1})\in A_{\ell}} (a_{n_{\ell-1}+1}a_{n_{\ell-1}+2}\ldots a_{n_{\ell}}a_{n_{\ell}+1})^{-2s}\\
   &:= \prod_{\ell=K}^{k} \Gamma_{\ell}(s).
\end{align}
Next, we estimate the upper bound of $\Gamma_{\ell}(s)$. We divide the integers $n_{\ell-1}+1,\ldots,n_{\ell}$ into two parts:
$$\Delta_{\ell,0}:=\left\{n_{\ell-1}+2k: k\in \mathbb{Z},~ 1\le k\le \frac{n_{\ell}-n_{\ell-1}}{2}\right\},$$
and 
$$\Delta_{\ell,1}:=\left\{n_{\ell-1}+2k+1: k\in \mathbb{Z},~ 0\le k\le \frac{n_{\ell}-n_{\ell-1}-1}{2}\right\}.$$
If $(a_{n_{\ell-1}+1}, \ldots, a_{n_{\ell}+1})\in A_{\ell}$, then either
$$\frac{1-\varepsilon}{2}\varphi(n_{\ell})\le\sum_{j\in \Delta_{\ell,0}}a_ja_{j+1}\le (1+\varepsilon)n_{\ell}\varphi(n_{\ell}),$$
or
\begin{equation}\label{36}
  \frac{1-\varepsilon}{2}\varphi(n_{\ell})\le\sum_{j\in \Delta_{\ell,1}}a_ja_{j+1}\le (1+\varepsilon)n_{\ell}\varphi(n_{\ell}).
\end{equation} 
Consider the case where $n_{\ell-1}$ and $n_{\ell}$ are even and $j\in \Delta_{\ell,1}$. The proof of other cases is similar. In this case, we have
\begin{equation}\label{37}
  \#\Delta_{\ell,1}=\frac{n_{\ell}-n_{\ell-1}}{2}.
\end{equation}
Let $b_j=a_ja_{j+1}$. We have
\begin{equation}\label{eq4.1}
    \prod_{j\in \Delta_{\ell,1}}b_j=a_{n_{\ell-1}+1}a_{n_{\ell-1}+2}\ldots a_{n_{\ell}}.
\end{equation}
From (\ref{36}), we deduce that
\begin{equation}\label{38}
  \frac{1-\varepsilon}{2}\varphi(n_{\ell})\le\sum_{j\in \Delta_{\ell,1}}b_j\le (1+\varepsilon)n_{\ell}\varphi(n_{\ell}).
\end{equation}
Set
$\pi(b_j)=\#\left\{(x,y)\in \mathbb{N}^2:xy=b_j\right\}.$
By Lemma \ref{yyl54}, 
\begin{equation}\label{eq4.3}
  \pi(b_j)\le c_{\varepsilon}b_j^{\varepsilon}.
\end{equation}
Define
$$D_{\ell}=\left\{i\in \mathbb{N}:\frac{1-\varepsilon}{2}\varphi(n_k)\le i\le (1+\varepsilon)n_k\varphi(n_k)\right\}\ \ 
\text{and}\ \ 
\Xi(\ell, m_{\ell})=\left\{(b_j)_{j\in \Delta_{\ell,1}}: \sum_{j\in \Delta_{\ell,1}}b_j=m_{\ell}\right\}.$$
Then by \eqref{37}, \eqref{eq4.1}, \eqref{38} and \eqref{eq4.3}, we have
\begin{align}\label{eq4.4}
    \Gamma_{\ell}(s)&\nonumber = \sum_{(a_{n_{l-1}+1}, \ldots, a_{n_{\ell}+1})\in A_{\ell}} (a_{n_{\ell-1}+1}a_{n_{\ell-1}+2}\ldots a_{n_{\ell}})^{-2s}\\
    &\nonumber \le \sum_{m_{\ell}\in D_{\ell}}\sum_{(b_j)\in \Xi(\ell, m_{\ell})} \prod_{j\in \Delta_{\ell,1}}\pi(b_j)b_j^{-2s}\\
    & \le c_{\varepsilon}^{\frac{n_{\ell}-n_{\ell-1}}{2}}\sum_{m_{\ell}\in D_{\ell}}\sum_{(b_j)\in \Xi(\ell, m_{\ell})}\prod_{j\in \Delta_{\ell,1}}b_j^{-2s+\varepsilon}.
\end{align}
 By applying Lemma \ref{LiaoRams2} to (\ref{eq4.4}), we obtain that
\begin{align}\label{eq4.6}
  \Gamma_{\ell}(s) 
   &\nonumber \le c_{\varepsilon}^{\frac{n_{\ell}-n_{\ell-1}}{2}}\sum_{m_{\ell}\in D_{\ell}}\left(\frac{9}{2}\left(2+\xi(2s-\varepsilon)\right)\right)^{\frac{n_{\ell}-n_{\ell-1}}{2}}m_{\ell}^{-2s+\varepsilon}\\
   &\nonumber \le c_{\varepsilon}^{\frac{n_{\ell}-n_{\ell-1}}{2}}\left(\frac{9}{2}(2+\xi(2s-\varepsilon))\right)^{\frac{n_{\ell}-n_{\ell-1}}{2}} \left(\frac{1-\varepsilon}{2}\varphi(n_{\ell})\right)^{-2s+\varepsilon}(1+\varepsilon)n_{\ell}\varphi(n_{\ell})
 \\
   & =Ce^{(1+\varepsilon-2s)\log\varphi(n_{\ell})+\log n_{\ell}+\frac{n_{\ell}-n_{\ell-1}}{2}c(s)},
\end{align}
where \[C=\big(\frac{1-\epsilon}{2}\big)^{-2s+\epsilon}(1+\epsilon)\ \ \text{and}\ \ c(s)=\log \left(\frac{9}{2}(2+\xi(2s-\varepsilon))c_{\epsilon}\right)\]
are independent of $\ell$.
By \eqref{eq3.21'} and \eqref{eq3.22'}, 
there exists $\ell_0(s)$ such that, when $\ell>\ell_0(s)$, we have
$$(1+\varepsilon-2s)\log\varphi(n_{\ell})+\log n_{\ell}+\frac{n_{\ell}-n_{\ell-1}}{2}c(s)<(1+\varepsilon-2s)\ell.$$
Hence,
\begin{equation}\label{eq4.7}
  Ce^{(1+\varepsilon-2s)\log\varphi(n_{\ell})+\log n_{\ell}+\frac{n_{\ell}-n_{\ell-1}}{2}c(s)}
<Ce^{(1+\varepsilon-2s)\ell}.
\end{equation}
Thus,  we deduce from \eqref{jk2}, \eqref{jk3}, \eqref{eq4.6} and \eqref{eq4.7} that 
\begin{align*}
  \mathcal{H}^s\left(J(\varphi,K)\right) & \le \liminf_{k\to\infty}\sum_{I_{n_k+1}\subseteq J(\varphi,k,K)}|I_{n_k+1}|^s  \le\liminf_{k\to\infty} \prod_{\ell=K}^{k} \Gamma_{\ell}(s)  
    \leq\liminf_{k\to\infty} \prod_{\ell=K}^{k} Ce^{(1+\varepsilon-2s)\ell}=0,
\end{align*}
which implies that $\dim_{\mathrm{H}}J(\varphi,K)\le 1/2$. Since $s\in(1/2,1)$ is arbitrary, we have $$\dim_{\mathrm{H}}L(\varphi)\le 1/2.$$
\subsection{\textbf{The case $1<\rho\leq\infty$}.}

In this section, we take $\phi(n):=\log\varphi(n)$. From Lemma \ref{i2} (2), we deduce that
\begin{equation}\label{eq5.2}
  \lim_{n\to\infty}\left(\phi(n+1)-\phi(n)\right)=\infty ~~\text{and}~~
\lim_{n\to\infty}\frac{\varphi(n+1)}{\varphi(n)}=\infty.
\end{equation}

\subsubsection{Lower bound}
To estimate the lower bound for the Hausdorff dimension of $L(\varphi)$, we will construct a Cantor-type subset $E=\bigcap\limits_{n\geq0}E_n$ 
contained in $L(\varphi)$. The Hausdorff dimension of $E$ will be computed in four steps by using the Lemma \ref{ex}.
\begin{enumerate}
    \item[\textbf{Step 1:}]Construct a Cantor-type subset of $L(\varphi)$. %$E\subseteq L(\varphi)$.
 Let $\{d_n\}$ be a sequence of positive real numbers, defined by 
\begin{equation}\label{d1}
 d_1=1,\ d_2=\varphi(1)~~ \text{and}~~  d_nd_{n+1}=\varphi(n)-\varphi(n-1), ~\forall\ n\geq2.
\end{equation}
By \eqref{eq5.2}, 
there exists an even number $N_1$ such that for $n\ge N_1$,
we have  $d_n\ge 2$, $\frac{d_n}{\phi(n-1)}\ge 3$, and
\begin{align*}
  \frac{d_{n+1}}{d_{n-1}} & =\frac{d_nd_{n+1}}{d_{n-1}d_n} 
    =\frac{e^{\phi(n)}-e^{\phi(n-1)}}{e^{\phi(n-1)}-e^{\phi(n-2)}}\\
   & =\frac{e^{\phi(n)}\left(1-e^{\phi(n-1)-\phi(n)}\right)}{e^{\phi(n-1)}\left(1-e^{\phi(n-2)-\phi(n-1)}\right)}
   = e^{\phi(n)-\phi(n-1)+o(1)},
\end{align*}
and
\begin{equation}\label{eq5.1}
  d_{n}d_{n+1}\ge \left(1+\frac{1}{\phi(n-2)}\right) \left(1+\frac{1}{\phi(n-1)}\right)d_{n-1}d_{n}.
\end{equation}
In the following, we claim that for any $n\ge N_1$,
\begin{equation}\label{d2}
  d_nd_{n+1}=e^{\phi(n)+o(n)}.
\end{equation} 
Indeed, if $n\ge N_1$ is even, then
\begin{equation*}
d_n= \frac{d_n}{d_{n-2}}\cdot\frac{d_{n-2}}{d_{n-4}}\cdots\frac{d_{N_1+4}}{d_{N_1+2}}\cdot\frac{d_{N_1+2}}{d_{N_1}} = e^{\phi(n-1)-\phi(n-2)+\cdots+\phi(N_1+3)-\phi(N_1+2)+\phi(N_1+1)-\phi(N_1)+o(n)}.
\end{equation*}
If $n\ge N_1$ is odd, then
\begin{equation*}
  d_n= \frac{d_n}{d_{n-2}}\cdot\frac{d_{n-2}}{d_{n-4}}\cdots\frac{d_{N_1+5}}{d_{N_1+3}}\cdot\frac{d_{N_1+3}}{d_{N_1+1}} = e^{\phi(n-1)-\phi(n-2)+\phi(n-3)-\phi(n-4)+\ldots+\phi(N_1+2)-\phi(N_1+1)+o(n)}.
\end{equation*}
Now, we use the sequence $\{d_n\}$ and the even number $N_1$ to construct a Cantor-type subset of $L(\varphi)$. Let
$$E=\left\{x\in(0,1): a_n(x)=1~\text{for}~1\leq n\le N_1,~d_n\le a_n(x)\le \left(1+\frac{1}{\phi(n-1)}\right)d_n~\text{for}~n> N_1\right\}.$$
 By \eqref{d2} and the definition of $L_n(x)$, we conclude that 
\[E\subseteq L(\varphi).\]

\item[\textbf{Step 2:}] Represent the subset $E$. For any $n\ge N_1$ and any positive integers $a_1,\ldots,a_n$, we define
$$J_{n}(a_1,\ldots,a_n):=\bigcup_{a_{n+1}} cl I_{n+1}(a_1,\ldots,a_n,a_{n+1}),$$
where “cl” denotes the closure of a set and the union is taken over all integers $a_{n+1}$ satisfying
$$ d_{n+1}\le a_{n+1}(x)\le \left(1+\frac{1}{\phi(n)}\right)d_{n+1}.$$
Let $a_i=1$ for all $i=1,\ldots, N_1$. For any $n\geq1$, define $E_0=[0,1]$ and
$$ E_n=\bigcup_{a_{N_1+1,\ldots,a_{N_1+n}}}J_{N_1+n}(a_1,\ldots,a_{N_1+n}),$$
 where the union is taken over all integers $a_{N_1+1},\ldots,a_{N_1+n}$ such that
$$ d_{N_1+i}\le a_{N_1+i}(x)\le \left(1+\frac{1}{\phi(N_1+i-1)}\right)d_{N_1+i},$$
for all $1\le i\le n$. Thus, we obtain
$$E=\bigcap_{n=0}^{\infty}E_n.$$

\item[\textbf{Step 3:}] Estimate the gap between $E_n$ and the number of $E_n$ contained in $E_{n-1}$. For any $n\geq1$, based on the structure of the set $E_n$, it is known that each $J_{N_1+n-1}(a_1,\ldots,a_{N_1+n-1})$ in $E_{n-1}$ contains at least $m_n$ intervals 
$J_{N_1+n}(a_1,\ldots,a_{N_1+n})$ of $E_{n}$. The number $m_n$ can be estimated as follows:
\begin{equation}\label{eq5.3}
 m_n=\left\lfloor\left(1+\frac{1}{\phi(N_1+n-1)}\right)d_{N_1+n}\right\rfloor-\lfloor d_{N_1+n}\rfloor\geq\frac{d_{N_1+n}}{\phi(N_1+n-1)}-1.
\end{equation}
Let $J_{N_1+n}(\tau_1,\ldots,\tau_{N_1+n})$ and $J_{N_1+n}(\sigma_1,\ldots,\sigma_{N_1+n})$ be two distinct intervals in $E_n$,  These intervals are separated by the basic interval of order $N_1+n+1$, namely, $I_{N_1+n+1}(\tau_1,\ldots,\tau_{N_1+n},1)$ or $I_{N_1+n+1}(\sigma_1,\ldots,\sigma_{N_1+n},1)$, depending on the relative position between 
 $J_{N_1+n}(\tau_1,\ldots,\tau_{N_1+n})$ and $J_{N_1+n}(\sigma_1,\ldots,\sigma_{N_1+n})$. Then by \eqref{lengthIn}, the gap between $J_{N_1+n}(\tau_1,\ldots,\tau_{N_1+n})$ and $J_{N_1+n}(\sigma_1,\ldots,\allowbreak\sigma_{N_1+n})$ is at least 
\begin{align}\label{eq5.4}
\nonumber |I_{N_1+n+1}(\tau_1,\ldots,\tau_{N_1+n},1)|
&\geq2^{-2(N_1+n+2)}(\tau_{N_1+1}\cdots \tau_{N_1+n})^{-2}\\
&\geq 2^{-2(N_1+n+2)} \prod_{i=1}^{n}\left(\left(1+\frac{1}{\phi(N_1+i-1)}\right)d_{N_1+i}\right)^{-2}:=\theta_n.
\end{align}
Note that $0<\theta_{n+1}<\theta_{n}$ for any $n\geq1$. A similar calculation yields the same inequality for the estimate of $|I_{N_1+n+1}(\sigma_1,\ldots,\sigma_{N_1+n},1)|$.

\item[\textbf{Step 4:}] Estimate the Hausdorff dimension of $E$. We distinguish the following two cases: $n=2k-1$ and $n=2k$ for any $k\geq1$.

  \textbf{Case 1:} If $n=2k-1$ for any $k\geq1$. Then, by \eqref{d2} and \eqref{eq5.3}, we have
      \begin{align*}
        m_1\cdots m_{n-1}&=m_1\cdots m_{2k-2}\\
        &\ge\prod_{i=1}^{2k-2} \left(\frac{d_{N_1+i}}{\phi(N_1+i-1)}-1\right)\geq\prod_{i=1}^{2k-2} \frac{d_{N_1+i}}{2\phi(N_1+i-1)}\\
         & =\prod_{i=1}^{2k-2}\frac{1}{2\phi(N_1+i-1)} (d_{N_1+1}d_{N_1+2})(d_{N_1+3}d_{N_1+4})\ldots (d_{N_1+2k-3}d_{N_1+2k-2})\\
         & = e^{\phi(N_1+1)+\phi(N_1+3)+\ldots+\phi(N_1+2k-3)(1+o(1))}.
      \end{align*}
At the same time, we deduce from \eqref{d2}, \eqref{eq5.3} and \eqref{eq5.4} that 
      \begin{align*}
        \theta_n m_n&=\theta_{2k-1}m_{2k-1}\\&\geq 2^{-2(N_1+2k-1+2)}\frac{d_{N_1+2k-1}}{2\phi(N_1+2k-2)} \prod_{i=1}^{2k-1}\left(1+\frac{1}{\phi(N_1+i-1)}\right)^{-2} \prod_{i=1}^{2k-1}d_{N_1+i}^{-2}\\
        &=e^{\phi(N_1+1)+\phi(N_1+2)\cdots+\phi(N_1+2k-2)(1+o(1))}.
      \end{align*}
Therefore, by Lemma \ref{ex}, we have
      \begin{align}\label{odd1}
      \dim_{\mathrm{H}}E & \ge \liminf_{n\to\infty}\frac{\log(m_1\cdots m_{n-1})}{-\log(\theta_nm_n)}=\liminf_{k\to\infty}\frac{\log(m_1\cdots m_{2k-2})}{-\log(\theta_{2k-1}m_{2k-1})}\nonumber\\
      &\geq\liminf_{k\to\infty} \frac{\sum_{i=1}^{k-1}\phi(N_1+2i-1)}{\sum_{i=1}^{2k-2}\phi(N_1+i)}\geq\liminf_{k\to\infty} \frac{\sum_{i=1}^{k}\phi(2i-1)}{\sum_{i=1}^{2k}\phi(i)}\nonumber\\
      &=\liminf_{k\to\infty} \frac{\sum_{i=1}^{k}\log\varphi(2i-1)}{\sum_{i=1}^{k}\log\varphi(2i-1) +\sum_{i=1}^{k}\log\varphi(2i)}\nonumber\\
      &=\frac{1}{1+\limsup_{k\to\infty}\frac{\sum_{i=1}^{k}\log\varphi(2i)}{\sum_{i=1}^{k}\log\varphi(2i-1)}}.
     \end{align}  

  \textbf{Case 2:} If $n=2k$ for any $k\geq1$. Then, by using the same methods in \textbf{Case 1}, we obtain that
  \[m_1\cdots m_{n-1}=m_1\cdots m_{2k-1}\geq e^{(\phi(N_1+2)+\phi(N_1+4)+\ldots+\phi(N_1+2k-2))(1+o(1))},\]
  and 
  \[\theta_n m_n=\theta_{2k}m_{2k}\geq e^{(\phi(N_1+1)+\phi(N_1+2)+\cdots+\phi(N_1+2k-1)(1+o(1))}.
\]
Thus, by Lemma \ref{ex}, we have
      \begin{align}\label{even1}
      \dim_{\mathrm{H}}E & \ge \liminf_{n\to\infty}\frac{\log(m_1\cdots m_{n-1})}{-\log(\theta_nm_n)}=\liminf_{k\to\infty}\frac{\log(m_1\cdots m_{2k-1})}{-\log(\theta_{2k}m_{2k})}\nonumber\\
      &\geq\liminf_{k\to\infty} \frac{\sum_{i=1}^{k-1}\phi(N_1+2i)}{\sum_{i=1}^{2k-1}\phi(N_1+i)}\geq\liminf_{k\to\infty} \frac{\sum_{i=1}^{k}\phi(2i)}{\sum_{i=1}^{2k+1}\phi(i)}\nonumber\\
      &=\liminf_{k\to\infty} \frac{\sum_{i=1}^{k}\log\varphi(2i)}{\sum_{i=1}^{k+1}\log\varphi(2i-1) +\sum_{i=1}^{k}\log\varphi(2i)}\nonumber\\
      &=\frac{1}{1+\limsup_{k\to\infty}\frac{\sum_{i=1}^{k+1}\log\varphi(2i-1)}{\sum_{i=1}^{k}\log\varphi(2i)}}.
      %&=\frac{1}{1+\limsup_{n\to\infty}\frac{\log\varphi(n+1)+\log\varphi(n-1)+\cdots+\log\varphi(n+1-2\lfloor n/2\rfloor)}{\log\varphi(n)+\log\varphi(n-2)+\cdots+\log\varphi(n-2\lfloor(n-1)/2\rfloor)}}.
\end{align}   
We deduce from \eqref{odd1} and \eqref{even1} that
$$\dim_{\mathrm{H}}L(\varphi)\geq \dim_{\mathrm{H}} E\geq\frac{1}{1+\beta},$$
where $\beta=\limsup\limits_{n\to\infty}\frac{\log\varphi(n+1)+\log\varphi(n-1)+\cdots+\log\varphi(n+1-2\lfloor n/2\rfloor)}{\log\varphi(n)+\log\varphi(n-2)+\cdots+\log\varphi(n-2\lfloor(n-1)/2\rfloor)}$.
\end{enumerate}

\subsubsection{Upper bound}

We will present a cover of the set $L(\varphi)$.
By Lemma \ref{i2} (2) and the definition of $L(\varphi)$, for any $0<\varepsilon<1/3$ and sufficiently large $n$, we have 
\begin{equation}\label{ff}
\frac{\varphi(n+1)}{\varphi(n)}>\frac{1+\varepsilon}{1-\varepsilon}\ \ \text{and}\ \ 
1-\varepsilon<\frac{L_n(x)}{\varphi(n)}< 1+\varepsilon.
\end{equation}
Combining \eqref{ff} with the definition of $L_n(x)$, we obtain 
$$a_n(x)a_{n+1}(x)\le L_n(x)< (1+\varepsilon)\varphi(n)\ \ \text{for sufficiently large $n$}.$$
We claim that
\[a_n(x)a_{n+1}(x)>(1-\varepsilon)\varphi(n)\ \ \text{for sufficiently large $n$}.\] 
Indeed, we deduce from \eqref{ff} that 
$$L_{n-1}(x)\le (1+\varepsilon)\varphi(n-1)<(1-\varepsilon)\varphi(n)< L_n(x)=\max\{L_{n-1}(x), a_n(x)a_{n+1}(x)\},$$
which implies $a_n(x)a_{n+1}(x)=L_n(x)>(1-\varepsilon)\varphi(n)$ for sufficiently large $n$. Clearly, we have
$$L(\varphi)\subseteq \bigcup_{N=1}^{\infty}E(\varphi,N),$$
where $E(\varphi,N)$ is defined as
$$E(\varphi,N):=\{x\in (0,1):(1-\varepsilon)\varphi(n)< a_n(x)a_{n+1}(x)<(1+\varepsilon)\varphi(n), \forall n\ge N\}.$$
It suffices to estimate the upper bound for the Hausdorff dimension of $E(\varphi, N)$ for all $N\ge 1$. We only consider the case $N=1$, the same method can be used in other cases. For any $n\ge 1$, set
$$D_{n+1}(\varphi):=\left\{(\sigma_1,\ldots,\sigma_{n+1})\in \mathbb{N}^{n+1}: (1-\varepsilon)\varphi(k)< \sigma_k\sigma_{k+1}<(1+\varepsilon)\varphi(k), \forall 1\le k\le n\right\}.$$
For any $(\sigma_1,\ldots,\sigma_{n+1})\in D_{n+1}(\varphi)$, let
\begin{equation}\label{jn2}
J_{n+1}(\sigma_1,\ldots,\sigma_{n+1}):=\bigcup_{\sigma_{n+2}:~(1-\varepsilon)\varphi(n+1)< \sigma_{n+1}\sigma_{n+2}<(1+\varepsilon)\varphi(n+1)} I_{n+2}(\sigma_1,\ldots,\sigma_{n+1},\sigma_{n+2}).
\end{equation}
Then, we have
\begin{equation}\label{eps}
E(\varphi,1)=\bigcap_{n=1}^{\infty}\bigcup_{(\sigma_1,\ldots,\sigma_{n+1})\in D_{n+1}(\varphi)}J_{n+1}(\sigma_1,\ldots,\sigma_{n+1}).
\end{equation}
For any $(\sigma_1,\ldots,\sigma_{n+1})\in D_{n+1}(\varphi)$, we shall estimate the length of $J_{n+1}(\sigma_1,\ldots,\sigma_{n+1})$ and the cardinality of the set $D_{n+1}(\varphi)$. 
It follows from \eqref{lengthIn} and \eqref{jn2} that
\begin{align}\label{delt2}
 \nonumber |J_{n+1}(\sigma_1,\ldots,\sigma_{n+1})| & \leq \sum_{\sigma_{n+1}\sigma_{n+2}>(1-\varepsilon)\varphi(n+1)} |I_{n+2}(\sigma_1,\ldots,\sigma_{n+1},\sigma_{n+2})| \\
\nonumber   & \le \sum_{\sigma_{n+1}\sigma_{n+2}>(1-\varepsilon)\varphi(n+1)}  \left(\frac{1}{\sigma_1\cdots\sigma_{n+1}\sigma_{n+2}}\right)^{-2}\\
\nonumber   &= \sum_{\sigma_{n+1}\sigma_{n+2}>(1-\varepsilon)\varphi(n+1)}\frac{1}{\sigma_1} \cdot\frac{1}{\sigma_1\sigma_2}\cdot\frac{1}{\sigma_2\sigma_3}\ldots \frac{1}{\sigma_{n}\sigma_{n+1}}\cdot\frac{1}{\sigma_{n+1}}\cdot\frac{1}{\sigma^{2}_{n+2}}\\
  \nonumber & \le \left(\frac{1}{1-\varepsilon}\right)^{n} \frac{1}{\varphi(1)\varphi(2)\cdots\varphi(n)} \frac{1}{\sigma_{n+1}} \sum_{\sigma_{n+1}\sigma_{n+2}>(1-\varepsilon)\varphi(n+1)}\frac{1}{\sigma^{2}_{n+2}}\\
   &\le \left(\frac{1}{1-\varepsilon}\right)^{n+1} \frac{2}{\varphi(1)\varphi(2)\cdots\varphi(n)\varphi(n+1)}:=\delta_{n+1}.
\end{align}
For the cardinality of the set $D_{n+1}(\varphi)$, we have
\begin{equation}\label{dn2}
  \#D_{n+1}(\varphi)  \le \sum_{\sigma_1=1}^{(1+\varepsilon)\varphi(1)} \sum_{\sigma_2=\frac{(1-\varepsilon)\varphi(1)}{\sigma_1}}^{\frac{(1+\varepsilon)\varphi(1)}{\sigma_1}} \sum_{\sigma_3=\frac{(1-\varepsilon)\varphi(2)}{\sigma_2}}^{\frac{(1+\varepsilon)\varphi(2)}{\sigma_2}}
  \cdots \sum_{\sigma_{n+1}=\frac{(1-\varepsilon)\varphi(n)}{\sigma_{n}}}^{\frac{(1+\varepsilon)\varphi(n)}{\sigma_{n}}}1.
\end{equation}
Notice that for any $k\geq 1$,
\begin{equation}\label{gn}
\sum_{\sigma_{k+1}=\frac{(1-\varepsilon)\varphi(k)}{\sigma_{k}}}^{\frac{(1+\varepsilon)\varphi(k)}{\sigma_{k}}} \sum_{\sigma_{k+2}=\frac{(1-\varepsilon)\varphi(k+1)}{\sigma_{k+1}}}^{\frac{(1+\varepsilon)\varphi(k+1)}{\sigma_{k+1}}}1 =\sum_{\sigma_{k+1}=\frac{(1-\varepsilon)\varphi(k)}{\sigma_{k}}}^{\frac{(1+\varepsilon)\varphi(k)}{\sigma_{k}}} \frac{2\varepsilon \varphi(k+1)}{\sigma_{k+1}}
\leq(2\varepsilon)^{2}(1-\varepsilon)^{-1}\varphi(k+1).
\end{equation}
To continue the proof, we distinguish the two cases.

\begin{enumerate}
    \item[\textbf{Case 1:}] If $n=2k-1$ for any $k\geq 1$. Then by \eqref{dn2} and \eqref{gn}, we have
\[\#D_{n+1}(\varphi)=\#D_{2k}(\varphi)\leq (2\varepsilon)^{2k-1}(1-\varepsilon)^{-k}\varphi^2(1)\varphi(3)\cdots\varphi(2k-1).\]
We deduce from \eqref{eps}, \eqref{delt2} and Lemma \ref{sjgh} that 
\begin{align}\label{odd2}
  \dim_{\mathrm{H}}E(\varphi,1) &\leq \liminf_{n\to\infty}\frac{\log(\#D_{n+1}(\varphi))}{-\log \delta_{n+1}}\leq\liminf_{k\to\infty}\frac{\log(\#D_{2k}(\varphi))}{-\log \delta_{2k}}\nonumber\\
   & =\liminf_{k\to\infty} \frac{\sum_{i=1}^{k}\log\varphi(2i-1)}{\sum_{i=1}^{2k}\log\varphi(i)}
    =\liminf_{k\to\infty} \frac{\sum_{i=1}^{k}\log\varphi(2i-1)}{\sum_{i=1}^{k}\log\varphi(2i-1) +\sum_{i=1}^{k}\log\varphi(2i)}\nonumber\\
      &=\frac{1}{1+\limsup_{k\to\infty}\frac{\sum_{i=1}^{k}\log\varphi(2i)}{\sum_{i=1}^{k}\log\varphi(2i-1)}}.
     \end{align}
     \item[\textbf{Case 2:}]  If $n=2k$ for any $k\geq 1$. Then by the same method used in \textbf{Case 1}, we obtain
\[\#D_{n+1}(\varphi)=\#D_{2k+1}(\varphi)\leq (2\varepsilon)^{2k}(1-\varepsilon)^{-k}(1+\varepsilon)\varphi(1)\varphi(2)\varphi(4)\cdots\varphi(2k).\]
Then, we have 
\begin{align}\label{even2}
  \dim_{\mathrm{H}}E(\varphi,1) &\leq \liminf_{n\to\infty}\frac{\log(\#D_{n+1}(\varphi))}{-\log \delta_{n+1}}\leq\liminf_{k\to\infty}\frac{\log(\#D_{2k+1}(\varphi))}{-\log \delta_{2k+1}}\nonumber\\
   & =\liminf_{k\to\infty} \frac{\sum_{i=1}^{k}\log\varphi(2i)}{\sum_{i=1}^{2k+1}\log\varphi(i)}
    =\liminf_{k\to\infty} \frac{\sum_{i=1}^{k}\log\varphi(2i)}{\sum_{i=1}^{k+1}\log\varphi(2i-1) +\sum_{i=1}^{k}\log\varphi(2i)}\nonumber\\
      &=\frac{1}{1+\limsup_{k\to\infty}\frac{\sum_{i=1}^{k+1}\log\varphi(2i-1)}{\sum_{i=1}^{k}\log\varphi(2i)}}.
\end{align}
\end{enumerate}
Thus, by \eqref{odd2} and \eqref{even2}, we conclude that
$$\dim_{\mathrm{H}}L(\varphi)\le \sup_{N\ge 1}\left\{\dim_{\mathrm{H}}E(\varphi,N)\right\}\le \frac{1}{1+\beta},$$
where $\beta=\limsup\limits_{n\to\infty}\frac{\log\varphi(n+1)+\log\varphi(n-1)+\cdots+\log\varphi(n+1-2\lfloor n/2\rfloor)}{\log\varphi(n)+\log\varphi(n-2)+\cdots+\log\varphi(n-2\lfloor(n-1)/2\rfloor)}.$

\section{Proof of Theorem \ref{mainthm3}}
\subsection{The case $\log\varphi(n)=\sqrt{n}/R(n)$.}
In this case, we will prove $\dim_{\mathrm{H}}L(\varphi)=1$. For any $k\ge 1$, let $n_k=k^2$, $s_k=\lfloor\varphi(n_k)\rfloor=\left\lfloor e^{\frac{k}{R(k^2)}}\right\rfloor$ and $t_k=\lfloor\varphi(n_k)\rfloor/k$. Then we obtain
\begin{equation}\label{eq13}
    \limsup_{k\to\infty}\frac{1}{k}\log \varphi(k^2)\le\limsup_{k\to\infty}\frac{1}{k}\frac{k}{R(k^2)}=\limsup_{k\to\infty}\frac{1}{R(k^2)}=0.
\end{equation}
Define
 $E(\{n_k\},\{s_k\},\{t_k\})$ as in \eqref{subset1}. By Lemma \ref{i2} (3), we have
$$\lim_{n\to\infty}\left(\log\varphi(m_{k+1})-\log \varphi(m_k)\right)=0~~\text{and}~~\lim_{n\to\infty}\frac{\varphi(m_{k+1})}{\varphi(m_k)}=1.$$ 
Hence, we obtain $$\lim\limits_{n\to\infty}\frac{L_n(x)}{\varphi(n)}=1,$$
which implies that
$E(\{n_k\},\{s_k\},\{t_k\})\subseteq L(\varphi).$
Using the same method as in  Theorem \ref{main thm1} for the case $0\le \rho<1/2$ and \eqref{eq13}, we obtain
$$\dim_{\mathrm{H}}L(\varphi)\ge \dim_{\mathrm{H}}E(\{n_k\},\{s_k\},\{t_k\})= 1.$$
We get the desired result.

\subsection{The case $\log\varphi(n)=\sqrt{n}R(n)$.}

For the lower bound of $\dim_{\mathrm{H}}L(\varphi)$,
 let $s_n=\sqrt{\varphi(n)}$ and $t_n=\sqrt{\varphi(n)/n}$, define
 $B(\{s_n\}, \{t_n\}, 1)$ as in (\ref{Bst}). Then, we have
  $$B\left(\{s_n\}, \{t_n\}, 1\right)\subseteq L(\varphi).$$
We can apply the same method as in Theorem \ref{main thm1} for the case $1/2<\rho<1$. Since this follows as a corollary of Lemma \ref{LemmaLR}, the proof is omitted.

For the upper bound, we follow the proof of Theorem \ref{main thm1} for the case $1/2<\rho\le 1$. For any $k\ge 1$, define $n_k=k^2$. From Lemma \ref{i2} (3), we deduce that $$\lim\limits_{k\to\infty}\frac{\varphi(n_{k+1})}{\varphi(n_k)}=\infty.$$ 
Let $x\in L(\varphi)$ be fixed. By applying the same arguments as in Proposition \ref{TandS}, we obtain that for sufficiently large $k\in\mathbb{N}$, there exists $j_k\ge 1$ such that $n_{k-1}<j_k\leq n_k$ and $L_{n_k}(x)=a_{j_k}(x)a_{j_k+1}(x)$. Consequently,  for any $s\in (1/2,1)$ and $0<\varepsilon<2s-1$, there exists $K\ge 1$ such that for all $k\ge K$,
\begin{equation}
   (1-\varepsilon)\varphi(n_k)\le L_{n_k}(x)\le S_{n_k}(x)-S_{n_{k-1}}(x)<S_{n_k}(x)\le (1+\varepsilon)\varphi(n_k). 
\end{equation}
 For any $k\ge K$, set
$$ \widetilde{M}_k=\big\{i\in \mathbb{N}:(1-\varepsilon)\varphi(n_k)\le i\le (1+\varepsilon)n_k\varphi(n_k)\big\}.$$
For any $K_1\ge K$, define
$$\widetilde{J}(\varphi,k, K_1)= \Big\{I_{n_k+1}(a_1,\ldots,a_{n_k+1}): \sum_{j=n_{\ell-1}+1}^{n_{\ell}}a_ja_{j+1}=m_{\ell} ~\text{with}~m_{\ell}\in M_{\ell},~K_1\le \ell \le k\Big\}.$$
%and
%$$J(\varphi, K_1)=\bigcap_{K_1=K_1}^{\infty}J(\varphi,k, K_1).$$
It follows that 
$$L(\varphi)\subseteq\bigcup_{K_1=K}^{\infty}\bigcap_{k=K_1}^{\infty}\widetilde{J}(\varphi,k, K_1).$$
As in the proof of the case $1/2<\rho\le 1$ in Theorem \ref{main thm1}, we estimate the sum
\begin{align*}
    \sum_{I_{n_k+1}\subseteq\widetilde{J}(\varphi,k,K)}|I_{n_k+1}|^s
    &\le \prod_{\ell=K}^k\left\{Cn_{\ell}e^{(1+\varepsilon-2s)\log\varphi(n_{\ell})}\left(\frac{9}{2}(2+\xi(2s-\varepsilon))\right)^{\frac{n_{\ell}-n_{\ell-1}}{2}}\right\}\\
    &= \prod_{\ell=K}^k\left\{C\ell^2e^{(1+\varepsilon-2s)\ell R(\ell^2)}\left(\frac{9}{2}(2+\xi(2s-\varepsilon))\right)^{\frac{2\ell-1}{2}}\right\},
\end{align*}
where $C$ is a constant independent of $\ell$. Since $R(x)$ is a regularly increasing function with index 0, by using the same method as in \eqref{eq4.7}, we can conclude that
\begin{equation*}
 \ell^2e^{(1+\varepsilon-2s)\ell R(\ell^2)}\left(\frac{9}{2}(2+\xi(2s-\varepsilon))\right)^{\frac{2\ell-1}{2}}
<e^{(1+\varepsilon-2s)\ell},
\end{equation*}
for sufficiently large $\ell$. This implies
$$\liminf_{k\to\infty}\sum_{I_{n_k+1}\subseteq\widetilde{J}(\varphi,k,K)}|I_{n_k+1}|^s=0.$$
Hence, we conclude that
$$\dim_{\mathrm{H}}L(\varphi)\le 1/2.$$

\section{Proof of Theorem \ref{mainthm2}}

\subsection{Lower bound}

We use Lemma \ref{fangkey} to construct the Cantor-type subset of $L(\psi)$. For any $k\in \mathbb{N}$, let 
\begin{equation}\label{zzn}
n_k=k^2,~~ s_k=e^{ck}~~\text{and}~~ t_k=\frac{s_k}{k}.
\end{equation}
Then, for any $n\ge 1$, there exists $k\in \mathbb{N}$ such that $n_k\le n< n_{k+1}$. Thus, $\psi(n)=e^{c\lfloor \sqrt{n}\rfloor}=e^{ck}$.
\begin{proposition}
For the above sequences $\{n_k\}, \{s_k\}\ \text{and}\ \{t_k\}$, we have 
  $$E\left(\{n_k\}, \{s_k\}, \{t_k\}\right)\subseteq L(\psi).$$
\end{proposition}
\begin{proof}
  Let $a_j(x)=1$ for any $j\neq n_k$ with $k\ge 1$. Then, for any $x\in E\left(\{n_k\}, \{s_k\}, \{t_k\}\right)$, we have
  $$e^{ck}=s_k\le a_{n_k}(x)\le s_k+t_k= \left(1+\frac{1}{k}\right)e^{ck}.$$
  For sufficiently large $n$, there exists $k$ such that
  $n_k\le n<n_{k+1}$, which implies that
  $$e^{ck}\leq a_{n_k}(x)a_{n_k+1}(x)=L_n(x)\leq\left(1+\frac{1}{k}\right)e^{ck}.$$
  Therefore,
 $$\lim_{n\to\infty}\frac{L_n(x)}{\psi(n)}=1.$$
\end{proof}
By \eqref{zzn} and \textbf{(H3)}, it can be checked that $\alpha_1=\frac{c}{2}\ \ \text{and}\ \ \alpha_2=0$. Then by Lemma \ref{fangkey}, we have
$$\dim_{\mathrm{H}}L(\psi)\ge \dim_{\mathrm{H}}E\left(\{n_k\}, \{s_k\}, \{t_k\}\right)=\theta_1(c/2,0)=\theta(c).$$

\subsection{Upper bound}
By classifying the value of the product of consecutive partial quotients, we shall construct a big Cantor-type set containing $L(\psi)$.
For any $c>0$ and integer $m\geq0$, let 
\[\Pi_{n}^{(m)}(x):=\prod_{a_i(x)a_{i+1}(x)> e^m\atop 1\le i\le n}a_i(x)a_{i+1}(x)\ \ \text{and}\ \
\Gamma_{m,n}(c)=\left\{x\in(0,1):\Pi_{n}^{(m)}(x)> e^{\frac{cn}{2}}\right\}.\]
Then, denote by 
\begin{equation}\label{gma}
\Gamma(c):=\bigcap\limits_{m=0}^{\infty}\Gamma_{m}(c).
\end{equation}
Here, the set
$\Gamma_m(c)$ is given by
\begin{equation}\label{tmc}
\Gamma_m(c):=\bigcap\limits_{N=1}^{\infty}\bigcup\limits_{n=N}^{\infty}\Gamma_{m,n}(c).
\end{equation}
  \begin{proposition}\label{X(c)2}
    For any $\delta>0$, we have $L(\psi)\subseteq \Gamma(c-\delta)$.
  \end{proposition}
  \begin{proof}
  Let $x\in L(\psi)$. For any $0<\varepsilon<\frac{e-1}{e}$ and $m\ge 0$, there exists $K_0(\varepsilon, m)$ such that for any $k\ge K_0$, 
  \begin{equation}\label{kk2}
L_{k^2}(x)\ge (1-\varepsilon)e^{ck}\ge e^m.
\end{equation}
    Using the same method as in Proposition \ref{TandS}, we can verify that
    $$L_{k^2}(x)=a_{j_k}(x)a_{j_k+1}(x),$$
where $(k-1)^2< j_k\le k^2.$ From \eqref{kk2}, we deduce that 
    $$\Pi_{k^2}^{(m)}(x)  \ge \prod_{i=K_0}^{k} L_{i^2}(x) \ge (1-\varepsilon)^{k-K_0+1}e^{\frac{c}{2}(k^2-K_0^2)} \ge e^{\frac{c}{2}(k^2-K_0^2)-k} \ge e^{\frac{(c-\delta)}{2}k^2},$$
    where $\delta$ depends on $K_0$ and the penultimate inequality holds for $(1-\varepsilon)^k\ge e^{-k}$. This implies that
    $$x\in \Gamma(c-\delta).$$  
  \end{proof}
 
 In the following, we shall estimate the upper bound of $\dim_{\mathrm{H}}\Gamma(c)$. 
\begin{theorem}\label{X(c)1}
  Let $c>0$. Then
  $$\dim_{\mathrm{H}}\Gamma(c)\le \theta(c),$$
  where $\theta(c)$ is the unique real solution of the equation $P(\theta)=c\left(\theta-\frac{1}{2}\right)$.
\end{theorem}
\begin{proof}
For any $\varepsilon>0$, choose positive integer $m_*>\max\left\{\frac{c}{2}, e^{8}\right\}$ large enough such that
  \begin{equation}\label{eq6.5}
    \max\left\{ (em_*)^{1/m_*}, \left(2em_*/c\right)^{1/m_*}\right\}\le e^{\varepsilon}.
  \end{equation}  
Then by \eqref{gma}, we have 
\begin{equation}\label{gmaa}
\Gamma(c)\subseteq \Gamma_{m_*}(c).  
\end{equation}
  It is sufficient to estimate the upper bound of the Hausdorff dimension of $\Gamma_{m_*}(c)$. By \eqref{tmc}, we first focus on the set $\Gamma_{m_*,n}(c)$. 
  Since $c>0$, there exists $N_0\in \mathbb{N}$ such that $e^{\frac{cn}{2}}>e^{m_*}$ for all $n\ge N_0$. Let $n\in \mathbb{N}$ with $n>N_0$ be  fixed. For any $x\in \Gamma_{m_*,n}(c)$, 
  there exists $1\le \ell \le \lfloor (n+1)/2\rfloor$
  with $1\le j_{\ell}\le n$ 
   and $j_k+1<j_{k+1}$ such that for all $1\le k\le \ell$, 
   \begin{equation}\label{eq6.3}
     a_{j_k}(x)a_{j_k+1}(x)> e^{m_*}\quad \text{and}\quad \prod_{k=1}^{\ell}a_{j_k}(x)a_{j_k+1}(x)> e^{\frac{cn}{2}}.
   \end{equation}
    Meanwhile, for all $1\le i\le n$ with $i\neq j_1,\ldots, j_{\ell}$, we have
     \begin{equation}\label{eq6.4}
       1\le a_i(x)a_{i+1}(x)\le e^{m_*}.
     \end{equation}       
     For any $1\le k \le \ell$, let $\lambda_k(x):=\lfloor\log a_{j_k}(x)a_{j_k+1}(x)\rfloor+1.$ 
     Then by (\ref{eq6.3}) and (\ref{eq6.4}), we have
     $$\lambda_1(x)+\ldots+\lambda_{\ell}(x) >\max\left\{\frac{cn}{2},m_*\ell\right\}\ \ \text{and}\ \ 
     e^{\lambda_k(x)-1} <a_{j_k}(x)a_{j_k+1}(x)\le e^{\lambda_k(x)}.$$
We take some notations. For any $n,\ell, \lambda\in\mathbb{N}$ with $1\le \ell\le \lfloor (n+1)/2\rfloor$
and $\lambda>\max\left\{cn/2, m_*\ell\right\}$, let
  $$\mathcal{A}_{n,\ell}:=\left\{(j_1,\ldots,j_{\ell})\in\mathbb{N}^{\ell}: j_k+1<j_{k+1}~\text{for all}~1\le k\le \ell ~\text{and}~1\le \ell\le \lfloor (n+1)/2\rfloor \right\},$$
  and
  $$\mathcal{B}_{\ell,\lambda} :=\left\{(\lambda_1,\ldots,\lambda_{\ell})\in\mathbb{N}^{\ell}:
    \lambda_1,\ldots,\lambda_{\ell}>m_*, ~ \lambda_1+\cdots+\lambda_{\ell}=\lambda\right\}.$$
  Let $\boldsymbol{j}_{\ell}:=(j_1,\ldots,j_{\ell})\in \mathcal{A}_{n,\ell}$ and $\boldsymbol{\lambda}_{\ell}:= (\lambda_1,\ldots,\lambda_{\ell})\in \mathcal{B}_{\ell,\lambda}$.
It follows that
  $$\Gamma_{m_*,n}(c)\subseteq\bigcup_{\ell=1}^{\lfloor(n+1)/2\rfloor}\bigcup_{\lambda>\max\{cn/2, m_*\ell\}}\bigcup_{\boldsymbol{j}_{\ell}\in\mathcal{A}_{n,\ell}}\bigcup_{ \boldsymbol{\lambda}_{\ell}\in \mathcal{B}_{\ell,\lambda}} \Gamma_{\boldsymbol{j}_{\ell}}^{\boldsymbol{\lambda}_{\ell}}(c),$$
  where 
  \begin{align*}
\Gamma_{\boldsymbol{j}_{\ell}}^{\boldsymbol{\lambda}_{\ell}}(c) :=\big\{x\in (0,1): &~1\le a_i(x)a_{i+1}(x)\le e^{m_*} ~\text{for all}~ 1\leq i\leq n ~\text{with}~ i\neq j_1,\ldots, j_{\ell}; \\
     &~e^{\lambda_k(x)-1}<a_{j_k}(x)a_{j_k+1}(x)\le e^{\lambda_k(x)} ~\text{for all}~ 1\le k\le \ell\big\}.
  \end{align*}
Now, we provide a symbolic description of the structure of $\Gamma_{m_*,n}(c)$. For any $n\ge 1$, let
  \begin{align}\label{cjn}
   \nonumber \mathcal{C}_{\boldsymbol{j}_{\ell}}^{\boldsymbol{\lambda}_{\ell}}(n+1) 
    :=\Big\{(\sigma_1,\ldots,\sigma_{n+1})\in \mathbb{N}^{n+1}:&~1\le \sigma_i\sigma_{i+1}\le e^{m_*} ~\text{for all}~ 1\leq i \leq n ~\text{with}~ i\neq j_1,\ldots, j_{\ell}; \\
     &~e^{\lambda_k(x)-1}<\sigma_{j_k}\sigma_{j_k+1}\le e^{\lambda_k(x)} ~\text{for all}~ 1\le k\le \ell \Big\}.
  \end{align} 
 Therefore,
 \begin{equation}\label{tmnc}
 \Gamma_{m_*,n}(c)\subseteq\bigcup_{\ell=1}^{\lfloor (n+1)/2\rfloor}\bigcup_{\lambda>\max\{cn/2, m_*\ell\}}\bigcup_{\boldsymbol{j}_{\ell}\in\mathcal{A}_{n,\ell}}\bigcup_{ \boldsymbol{\lambda}_{\ell}\in \mathcal{B}_{\ell,\lambda}}\bigcup_{(\sigma_1,\ldots,\sigma_{n+1})\in \mathcal{C}_{\boldsymbol{j}_{\ell}}^{\boldsymbol{\lambda}_{\ell}}(n+1)}I_{n+1}(\sigma_1,\ldots,\sigma_{n+1}).
\end{equation}

 The following is to estimate the cardinalities of $\mathcal{A}_{n,\ell}$ and $\mathcal{B}_{\ell,\lambda}$, as well as the diameter of $I_{n+1}(\sigma_1, \dots, \allowbreak\sigma_{n+1})$ with $(\sigma_1, \dots, \sigma_{n+1})\in \mathcal{C}_{\boldsymbol{j}_{\ell}}^{\boldsymbol{\lambda}_{\ell}}(n+1)$. 
 Before proceeding, we state a version of the Stirling formula (see \cite{R55}) that will be used in the sequel: 
 \begin{equation} \label{str}
 \sqrt{2\pi}n^{n+\frac{1}{2}}e^{-n}\leq n!\leq en^{n+\frac{1}{2}}e^{-n},\ \ \forall\ n\geq1.
 \end{equation}
 Let $n, \ell, \lambda \in \mathbb{N}$ be fixed such that $1 \leq \ell \leq \lfloor (n+1)/2\rfloor$ and $\lambda > \max \{ cn/2, m_* \ell \}$. Then by \eqref{eq6.5} and \eqref{str}, we have
 \begin{align}\label{Ano}
\nonumber\#\mathcal{A}_{n,\ell} \leq\begin{pmatrix} \lfloor (n+1)/2\rfloor - \ell \\ \ell \end{pmatrix}  <\begin{pmatrix} n  \\ \ell \end{pmatrix}
<\frac{n^{\ell}}{\ell !}<\frac{1}{\sqrt{2\pi\ell}}\left(\frac{en}{\ell}\right)^{\ell}&<\frac{1}{\sqrt{2\pi\ell}}\left(\frac{2e\lambda}{c\ell}\right)^{\ell}\\
&<\left(\frac{2e\lambda}{c\ell}\right)^{\ell}<\left(\frac{2}{c}em_*\right)^{\frac{\lambda}{m_*}}\le e^{\varepsilon \lambda},
\end{align}
where the fifth inequality holds for $\lambda>\frac{cn}{2}$, the penultimate inequality comes from the fact that $m_*>c/2$ and the function $\ell \mapsto \left( \frac{2e\lambda}{c\ell} \right)^{\ell}$ is increasing on $(0, 2\lambda/c)$.
For the cardinality of $\mathcal{B}_{\ell, \lambda}$, by using \eqref{eq6.5} and \eqref{str} again, we obtain
\begin{align}\label{Bnu}
\nonumber\#\mathcal{B}_{\ell, \lambda}=\begin{pmatrix} \lambda- m_* \ell-1 \\ \ell - 1 \end{pmatrix}<\begin{pmatrix} \lambda - 1 \\ \ell - 1 \end{pmatrix} < \frac{\lambda^{\ell}}{\ell!}
&< \frac{1}{\sqrt{2\pi\ell}}\left( \frac{e \lambda}{\ell} \right)^{\ell}\\
&< \left( \frac{e \lambda}{\ell} \right)^{\ell}<(e m_*)^{\lambda/m_*} \leq e^{\varepsilon \lambda},
\end{align}
where the penultimate inequality holds for $\ell < \lambda / m_*$ and the function $\ell \mapsto \left( \frac{e \lambda}{\ell} \right)^{\ell}$ is increasing on $(0, \lambda)$.
Now, we turn to estimate the diameter of $I_{n+1}(\sigma_1, \dots, \sigma_{n+1})$. 
For any $\boldsymbol{j}_{\ell}\in \mathcal{A}_{n,\ell}$, $\boldsymbol{\lambda}_{\ell}\in \mathcal{B}_{\ell,\lambda}$ and $(\sigma_1,\ldots,\sigma_{n+1})\in \mathcal{C}_{\boldsymbol{j}_{\ell}}^{\boldsymbol{\lambda}_{\ell}}(n+1)$, 
By Lemma \ref{lemma2.2}, \eqref{lengthIn} and \eqref{cjn}, we have
\begin{align}\label{eq6.8}
\nonumber |I_{n+1}(\sigma_1,\ldots,\sigma_{n+1})| &\le 2^{\ell} 8^{2\ell} \left(\prod_{k=1}^{\ell}\big|I_2(\sigma_{j_k},\sigma_{j_k+1})\big|\right) \big|{I}_{n+1-2\ell}(\tau_1,\ldots,\tau_{n+1-2\ell})\big| \\
\nonumber & \le 2^{7\ell} \left(\prod_{k=1}^{\ell}(\sigma_{j_k}\sigma_{j_k+1})^{-2}\right) \frac{1}{{q^{2}_{n+1-2\ell}}(\tau_1,\ldots, \tau_{n+1-2\ell})}  \\
\nonumber & \le 2^{7\ell} \left(\prod_{k=1}^{\ell}\left(e^{-2(\lambda_k(x)-1)}\right)\right) \frac{1}{{q^{2}_{n+1-2\ell}}(\tau_1,\ldots, \tau_{n+1-2\ell})} \\
&=(2^7e^2)^{\ell} e^{-2\lambda} \frac{1}{{q^{2}_{n+1-2\ell}}(\tau_1,\ldots, \tau_{n+1-2\ell})},
\end{align}
where $(\tau_1,\ldots, \tau_{n+1-2\ell})$ denotes the sequence obtained by  eliminating the terms $\{\sigma_{j_k}, \sigma_{j_k+1}: 1\le k\le \ell\}$ from $(\sigma_1,\ldots, \sigma_{n+1})$. That is, for all $1\leq i\leq n-2\ell$, we have $1\le \tau_i(x)\tau_{i+1}(x)\le e^{m_*}$.

In the following, we shall choose a suitable positive real number $s$ such that 
$\mathcal{H}^s\left(\Gamma_{m_*}(c)\right)\leq0$. It is worth pointing out that  $\theta(c)$ is the unique real solution of $P(\theta)=c\left(\theta-\frac{1}{2}\right)$, then $\theta(c)\in (1/2,1)$. For any $s>\theta(c)$, we deduce that $P(s)<c\left(s-\frac{1}{2}\right)$.  Let $\varepsilon$ be small enough such that \eqref{eq6.5} and 
\begin{equation}\label{epsilon2}
0<\varepsilon<\min\left\{\frac{2s-1}{s+2}, \frac{c(s-1/2)-P(s)}{(1+s/2)c+1}\right\}
\end{equation}
 hold. Denote by 
 \begin{equation}\label{Sig}
  \Sigma_s:=\sum_{(\sigma_1,\ldots,\sigma_{n+1})\in \mathcal{C}_{\boldsymbol{j}_{\ell}}^{\boldsymbol{\lambda}_{\ell}}(n+1)} |I_{n+1}(\sigma_1,\ldots,\sigma_{n+1})|^s.
  \end{equation}
  Then by \eqref{eq6.8}, we have
  $$\Sigma_s\le \sum_{(\sigma_1,\ldots,\sigma_{n+1})\in \mathcal{C}_{\boldsymbol{j}_{\ell}}^{\boldsymbol{\lambda}_{\ell}}(n+1)} (2^7e^2)^{\ell s} e^{-2\lambda s} \frac{1}{q^{2s}_{n+1-2\ell}(\tau_1,\ldots, \tau_{n+1-2\ell})}.$$
  Notice that 
  \begin{equation*}
  \begin{cases}
  1\le \sigma_i(x)\sigma_{i+1}(x)\le e^{m_*}\ \text{for all}~ 1\leq i \leq n ~\text{with}~ i\neq j_1,\ldots, j_{\ell}, \\
  e^{\lambda_k(x)-1}<\sigma_{j_k}(x)\sigma_{j_k+1}(x)\le e^{\lambda_k(x)} \ \text{for all}\ 1\le k\le \ell.
  \end{cases}
  \end{equation*}
  Then, we have
  \begin{align}\label{sigmg}
   \nonumber \Sigma_s & \le (2^7e^2)^{\ell s}\prod_{k=1}^{\ell}\left(\sum_{e^{\lambda_k-1}<\sigma_{j_k}\sigma_{j_k+1} \le e^{\lambda_k}}e^{-2\lambda s}\right)\sum_{1\le \tau_i\tau_{i+1}\le e^{m_*}}\frac{1}{{q^{2s}_{n+1-2\ell}}(\tau_1,\ldots, \tau_{n+1-2\ell})}\\
     & \le (2^7e^2)^{\ell s}e^{(1-2s)\lambda}\sum_{\tau_1,\ldots, \tau_{n+1-2\ell}\in \mathbb{N}} \frac{1}{{q^{2s}_{n+1-2\ell}}(\tau_1,\ldots, \tau_{n+1-2\ell})}.
  \end{align}
  Since $\lambda>m_*\ell$ and $m_*>e^{8}$, we deduce from (\ref{eq6.5}) that
  \begin{equation}\label{yyg}
   (2^7e^2)^{\ell}<(em_*)^{\lambda/m_*}\le e^{\varepsilon \lambda}.   
  \end{equation} 
Notice that $s>1/2$, then by \eqref{pressure function}, there exists $K_{\varepsilon}>0$ such that for all $n\ge 1$,
   \begin{equation}\label{DEP}
       \sum_{a_1,\ldots,a_n\in \mathbb{N}}q_n^{-2s}(a_1,\ldots,a_n) \le K_{\varepsilon}e^{n(P(s)+\varepsilon)}.
   \end{equation}
  Substituting \eqref{yyg} and \eqref{DEP} into \eqref{sigmg}, we obtain
  \begin{align*}
    \Sigma_s & \le e^{\varepsilon \lambda s}e^{(1-2s)\lambda }K_{\varepsilon}e^{(n+1-2\ell)(P(s)+\varepsilon)} \leq K_{\varepsilon}e^{\lambda(1-2s+\varepsilon s)}e^{n(P(s)+\varepsilon)}.
  \end{align*}
 This, in combination with \eqref{Ano} and \eqref{Bnu}, implies that 
  \begin{align}\label{gj4}
   \nonumber \sum_{\lambda>\max\{cn/2,m_*\ell\}}\sum_{\boldsymbol{j}_{\ell}\in \mathcal{A}_{n,\ell}}\sum_{\boldsymbol{\lambda}_{\ell}\in \mathcal{B}_{\ell,\lambda}}\Sigma_s 
     & \le K_{\varepsilon}e^{n(P(s)+\varepsilon)}\sum_{\lambda>\max\{cn/2,m_*\ell\}} e^{(1-2s+\varepsilon(s+2))\lambda} \\
   \nonumber & \le K_{\varepsilon}e^{n(P(s)+\varepsilon)}\sum_{\lambda>cn/2} e^{(1-2s+\varepsilon(s+2))\lambda} \\
    &\le K_{\varepsilon}^*e^{n(P(s)+\varepsilon)+(1-2s+\varepsilon(s+2))cn/2},
  \end{align} 
  where $K_{\varepsilon}^*$ is a constant only depending on $s$, $\varepsilon$ and $c$. 
Now we are ready to estimate the upper bound of the Hausdorff dimension of $\Gamma_{m_*}(c)$. 
It follows from \eqref{tmc}, \eqref{tmnc}, \eqref{epsilon2}, \eqref{Sig} and \eqref{gj4} that
 \begin{align*}
  \mathcal{H}^s\left(\Gamma_{m_*}(c)\right)&\leq \liminf_{N\to\infty}\sum_{n=N}^{\infty}\sum_{\ell=1}^{\lfloor (n+1)/2\rfloor} \sum_{\lambda>\max\{bn,m_*\ell\}}\sum_{\boldsymbol{j}_{\ell}\in \mathcal{A}_{n,\ell}}\sum_{\boldsymbol{\lambda}_{\ell}\in \mathcal{B}_{\ell,\lambda}}\sum_{(\sigma_1,\ldots,\sigma_{n+1})\in \mathcal{C}_{\boldsymbol{j}_{\ell}}^{\boldsymbol{\lambda}_{\ell}}(n+1)} |I_{n+1}(\sigma_1,\ldots,\sigma_{n+1})|^s\\
  &\leq\liminf_{N\to\infty}\sum_{n=N}^{\infty}\sum_{\ell=1}^{\lfloor (n+1)/2\rfloor} \sum_{\lambda>\max\{bn,m_*\ell\}}\sum_{\boldsymbol{j}_{\ell}\in \mathcal{A}_{n,\ell}}\sum_{\boldsymbol{\lambda}_{\ell}\in \mathcal{B}_{\ell,\lambda}}\Sigma_s\\
  &\leq\liminf_{N\to\infty}\sum_{n=N}^{\infty}\sum_{\ell=1}^{\lfloor (n+1)/2\rfloor}  K_{\varepsilon}^*e^{n(P(s)+\varepsilon)+(1-2s+\varepsilon(s+2))cn/2} \\
 & \le K_{\varepsilon}^*\liminf_{N\to\infty}\sum_{n=N}^{\infty} n e^{n(P(s)+\varepsilon)+(1-2s+\varepsilon(s+2))cn/2}=0.
  \end{align*} 
This shows that \[\dim_{\mathrm{H}}\Gamma_{m_*}(c)\le s.\]
Consequently, it follows from \eqref{gmaa} that $\dim_{\mathrm{H}} \Gamma(c) \leq s$. Since $s > \theta(c)$ is arbitrary, we conclude that
$$
\dim_{\mathrm{H}} \Gamma(c) \leq \theta(c).
$$

  \end{proof}
 
From Theorem \ref{X(c)1} and Proposition \ref{X(c)2}, we deduce that, for any $\delta>0$,
  $$\dim_{\mathrm{H}}L(\psi)\le \dim_{\mathrm{H}}\Gamma(c-\delta)\le \theta(c-\delta).$$
Letting $\delta\to 0$, we have
$$\dim_{\mathrm{H}}L(\psi)\le \theta(c).$$

{\bf Acknowledgement:}
The authors would like to thank Professor Lingmin Liao for his invaluable comments. The research is partially supported by Natural Science Foundation of Hunan Province
(No. 2025JJ60006) and National Natural Science Foundation of China (Nos.\, 12201207, 12371072).

\end{document}